\documentclass{ws-jhde}

\def\p{\partial}

\def\ul{\underline}

\usepackage{amsmath,amsfonts,amssymb}

\usepackage{mathrsfs}

\newtheorem{defi}{Definition} 
\newtheorem{lem}{Lemma}

\newcounter{subdefi}

\newcommand{\maxN}{4}

\begin{document}

\markboth{D. Hilditch and R. Richter}
{Hyperbolicity of high order systems}

%
\catchline{}{}{}{}{}
%

\title{HYPERBOLICITY OF HIGH ORDER SYSTEMS OF EVOLUTION EQUATIONS}

\author{DAVID HILDITCH}

\address{Theoretical Physics Institute, University of 
Jena,\\ 07743 Jena, Germany\\
\email{david.hilditch@uni-jena.de}
}

\author{RONNY RICHTER}

\address{Mathematisches Institut, Universi\"at T\"ubingen,\\ 
72076 T\"ubingen, Germany
}

\maketitle


\begin{abstract}
{\bfseries Abstract.}\quad We study properties of evolution equations 
which are first order in time and arbitrary order in space (FT$N$S).
Following Gundlach and Mart\'in-Garc\'ia~($2006$) we define 
strong and symmetric hyperbolicity for FT$N$S systems
and examine the relationship between these definitions, 
and the analogous concepts for first order systems. We 
demonstrate equivalence of the FT$N$S definition of strong 
hyperbolicity with the existence of a strongly hyperbolic 
first order reduction. We also demonstrate equivalence
of the FTNS definition, up to $N=4$, of symmetric hyperbolicity 
with the existence of a symmetric hyperbolic first order 
reduction.
\end{abstract}



\section{Introduction}
\label{section:Introduction}

Systems of partial differential equations admitting wave-like 
solutions are ubiquitous in both physics and applied mathematics.
With additional smoothness assumptions, it is known that by 
restricting to the special case with at most first order derivatives 
the initial value problem of such systems can be classified 
algebraically with respect to its well-posedness. The crucial step 
in this classification is to check for strong hyperbolicity by 
analyzing the principal part, i.e. the derivative terms, of the 
evolution system~\cite{GusKreOli95,KreLor89}. 

The theory used to demonstrate this relies on pseudo-differential 
calculus~\cite{Tay81}. By performing a pseudo-differential 
reduction to first order the basic method can also be applied
to evolution systems with higher order derivatives, see for 
example~\cite{NagOrtReu04,GunGar05}.

For the initial boundary value problem the theory is not so 
complete. The simplest approach for first order systems is to check 
for a stronger condition, called symmetric hyperbolicity. With 
carefully chosen boundary conditions it can be used to identify a 
well-posed initial boundary value problem~\cite{GusKreOli95,KreLor89}.
If the evolution system is not symmetric hyperbolic there is still 
hope to demonstrate well-posedness, e.g. by employing the 
Laplace-Fourier method~\cite{Kre70,Agr72,Met00,SarTig12}, which 
unfortunately does not apply to arbitrary strongly hyperbolic 
evolution systems.

We study strong and symmetric hyperbolicity for a special class
of higher order evolution equations. Hyperbolicity of higher
order systems was studied before in a different context,
see e.g.~\cite{Bei04,Wal84a,Chr00}. The equations of interest here 
are linear constant coefficient, first order in time and arbitrary 
order in space systems (FT$N$S). They admit a reduction to first
order for which standard definitions of hyperbolicity are applicable.

Reductions to first order are obtained by introducing new variables
for all but the highest order derivatives~\cite{Ger96},
which is a common approach in numerical
relativity~\cite{KidSchTeu01,SarCalPul02,BeySar04,LinSchKid05}.
In this way the known, first order definitions of hyperbolicity can be
applied, and powerful numerical methods are available in the
construction of approximate
solutions~\cite{GusKreOli95,KreLor89,Gus08}.

However, making the first order reduction raises questions,
e.g. about the number of constraints to impose and the size of
the approximation error~\cite{KrePetYst02,CalHinHus05}. For practical 
applications it also incurs a cost. The memory footprint of 
any numerical approximation method increases hugely due to the
auxiliary variables.

The question we address here is whether or not we can characterize 
hyperbolicity of FT$N$S systems without making a differential or
pseudo-differential reduction to first order. The idea is to establish
when ``good'' reductions of either type can be made. For the important
case of second order in space systems this question was already
answered satisfactorily in the affirmative by Gundlach and 
Mart\'in-Garc\'ia~\cite{GunGar05}, see also~\cite{GunGar04,GunGar04a,
GunGar06} for applications of these ideas. The present work is the 
extension of those calculations to first order in time, higher order
in space systems. The generalization here will be useful in analyzing 
higher derivative systems. A more abstract treatment of evolution 
systems can be found in~\cite{Bey05}.

We propose definitions of strong and symmetric hyperbolicity for
FT$N$S systems without reference to any first order system.
This enables us to demonstrate equivalence of FT$N$S strong
hyperbolicity with the existence of an iterative first 
order reduction, either differential or pseudo-differential, that is 
strongly hyperbolic in the sense of first order systems.

We also find that if a higher order system has a symmetric hyperbolic
first order reduction then the equations must satisfy the FT$N$S
definition of symmetric hyperbolicity. Conversely, for systems
containing up to fourth order spatial derivatives, we show that
the new definition of symmetric hyperbolicity is also sufficient
for the existence of a symmetric hyperbolic first order reduction.

The first order reduction used in this case is a direct, not
iterative method, i.e. it differs from the one
applied in the proofs concerning strong hyperbolicity.
As discussed in section \ref{sec:diff_reds},
the iterative, order-by-order reduction is not appropriate for
symmetric hyperbolicity.

The Laplace-Fourier method, which can be used to prove well-posedness
of initial boundary value problems is not considered here.
Higher order derivative evolution systems can be treated by 
this technique (see for example~\cite{KreOrtPet10}), because it 
once again relies on pseudo-differential calculus.

The paper is structured as follows. In 
section~\ref{section:Basic_Hyp} we review the definitions of 
strong and symmetric hyperbolicity for first order in time, 
second order in space systems. For pedagogical purposes, in 
section~\ref{section:Third_order}, we explicitly present the special 
case of the extension of the theory to first order in time, third order 
in space systems. Then we provide a general formulation of first 
order in time,~$N$-th order in space systems in 
section~\ref{section:Higher_order}. In section~\ref{section:FTNS_strong_hyp} 
we discuss strong hyperbolicity using an iterative reduction 
procedure. In section~\ref{section:FTNS_sym_hyp} definitions 
for symmetric hyperbolicity are given for the higher order system 
without reduction. The relationship between the definitions is 
then investigated using a direct reduction to first order.
We conclude in section~\ref{section:Conclusion}. 

\section{Basic notions of hyperbolicity}
\label{section:Basic_Hyp}

In this article we consider a special class of linear systems 
of partial differential equations with constant coefficients. 
We are mainly interested in questions about the well-posedness 
of initial (boundary) value problems.

\paragraph*{Well-posedness:} An initial (boundary) value 
problem is called well-posed if there is a unique solution that 
depends continuously, in some appropriate norm, on the choice 
of initial data.

\paragraph*{Second order systems:} The class of partial 
differential equations under consideration is a generalization 
of the first order in time, second order in space systems 
analyzed in~\cite{GunGar05,GunGar04,GunGar04a}. We start with 
a short summary of that work. Consider first order in time, 
second order in space systems of the form
\begin{align}
\label{eq:systems_FOITSOIS}
\p_t \tilde u=(A^u{}_u)^i\p_i\tilde u+A^u{}_v\tilde v+S_{u},\quad\quad
\p_t \tilde v=(A^v{}_u)^{ij}\p_i\p_j\tilde u
+(A^v{}_v)^i\p_i\tilde v+S_{v}.
\end{align}
where we have absorbed all non-principal terms into the source
functions $S$. They have the form
$S_{u} = \alpha_1 \tilde u + f_{u}$, and
$S_{v} = \alpha_2^i \p_i \tilde u + \alpha_3 \tilde u 
+ \alpha_4 \tilde v + f_{v}$,
where~$f_{u}$ and~$f_{v}$ do not depend on~$\tilde u$ 
or~$\tilde v$ and the~$\alpha_i$ are constant coefficient 
matrices.

\paragraph*{Principal part:}
The \emph{principal part} of the system~\eqref{eq:systems_FOITSOIS} is
\begin{align}
\p_t\tilde u\simeq (A^u{}_u)^i\p_i\tilde u + A^u{}_v\tilde v,\quad
\p_t\tilde v\simeq (A^v{}_u)^{ij}\p_i\p_j\tilde u+(A^v{}_v)^i\p_i\tilde v,\nonumber
\end{align}
where~$\simeq$ denotes equality up to non-principal terms.
We denote the matrix
\begin{align}\nonumber
\mathcal A_2^p{}_i{}^j = 
\left(\begin{array}{cc}
 (A^u{}_u)^j\delta^p{}_i & A^u{}_v\delta^p{}_i \\
 (A^v{}_u)^{pj}          & (A^v{}_v)^p
\end{array}\right)
\end{align}
the \emph{principal matrix} of the system~\eqref{eq:systems_FOITSOIS}.
For a fixed unit spatial vector $s^i$ the \emph{principal symbol} of the 
system~\eqref{eq:systems_FOITSOIS} is
\begin{align}\nonumber
P_2^s = \left(\begin{array}{cc}
 (A^u{}_u)^i s_i & A^u{}_v \\
 (A^v{}_u)^{ij} s_is_j & (A^v{}_v)^i s_i
\end{array}\right).  
\end{align}
Note that with~$S_i:=\mbox{diag}(s_i,1)$ one obtains the principal symbol 
from the principal matrix by the contraction~$P_2^s = S^i A_2^p{}_i{}^j S_j s_p$.
Furthermore the equations of motion for the variables~$\tilde{u}_i=(\p_i\tilde u,\tilde v
)^\dagger$, are up to non principal 
terms~$\p_t\tilde u_i\simeq A_2^p{}_i{}^j\p_p\tilde u_j$.

\paragraph*{Strong hyperbolicity:}
Following~\cite{GunGar05,SarTig12,GunGar04,GunGar04a} the 
system~(\ref{eq:systems_FOITSOIS}) 
is called strongly hyperbolic if there exist a constant $M_2>0$
and a family of hermitian matrices $H_2(s)$ such that
\begin{align}
\label{eq:FT2S_strong_hyp_def}
H_2(s) P_2^s = (P_2^s)^\dag H_2(s),\quad
M_2^{-1}\,I \leq H_2(s) \leq M_2\,I,
\end{align}
where we used the standard inequality for hermitian matrices
\begin{align}
\label{eq:matrix_inequality}
A\leq B \quad \Leftrightarrow \quad v^\dag A v \leq v^\dag B v\;\;\forall v.
\end{align}
It is a necessary and sufficient condition for well-posedness of 
the initial value problem. This definition is furthermore equivalent 
to the existence of a fully first order reduction 
of~(\ref{eq:systems_FOITSOIS}) which satisfies the standard definition 
of strong hyperbolicity for first order systems. 

Note that this is not quite equivalent to the definition
given in~\cite{GunGar05,GunGar04,GunGar04a}, where it is required that
the principal symbol has real eigenvalues and a complete set of
eigenvectors that depend continuously on $s$.

What can be shown \cite{SarTig12,KreLor89} is that \eqref{eq:FT2S_strong_hyp_def}
is equivalent to the existence of a constant $K_2>0$
and a family of matrices $T_2(s)$ such that
\begin{align}
\label{eq:FT2S_strong_hyp_eqiv_def}
T_2(s)^{-1} P_2^s T_2(s) &= \Lambda(s), &
K_2^{-1} &\leq \| T_2(s)\| \leq K_2,
\end{align}
with a real, diagonal matrix $\Lambda(s)$ and the standard
spectral norm $\|\cdot\|$.

In view of example~$12$ in~\cite{SarTig12}, the 
continuity of $T_2(s)$ required in~\cite{GunGar05,GunGar04,GunGar04a}
is sufficient
to guarantee the existence of~$K_2$, but not necessary. 
Fortunately despite the continuity condition being slightly too 
restrictive, the construction of first order reductions with the 
approach of~\cite{GunGar05} is unaltered if we instead 
require~\eqref{eq:FT2S_strong_hyp_eqiv_def}. Our treatment of 
strong hyperbolicity for FT$N$S systems is therefore the 
natural generalization of~\cite{GunGar05}.

\paragraph*{Symmetric hyperbolicity:} For the analysis of the 
initial \emph{boundary} value problem the stronger notion of 
symmetric hyperbolicity is desirable. It guarantees the 
existence of a conserved energy in the principal part and 
allows the construction of boundary conditions such that the 
initial boundary value problem is well posed. A Hermitian 
matrix
\begin{align}
\nonumber
H_2^{ij} &=
\left(\begin{array}{cc}
H_{11}^{ij} & H_{12}^i \\
H_{12}^{\dagger j} & H_{22}
\end{array}\right),
\end{align}
independent of $s^i$, such that the matrix~$S_iH_2^{ij}A_2^p{}_j{}^ks_pS_k$,
is Hermitian for every spatial vector $s^i$ is called a candidate 
symmetrizer. The system~\eqref{eq:systems_FOITSOIS} is called 
symmetric hyperbolic if there exists a positive definite candidate 
symmetrizer. The aforementioned conserved energy is
\begin{align}
\nonumber
E &= \int d^3x \; \epsilon = \int d^3x \; \tilde u_i^\dag H_2^{ij} \tilde u_j.
\end{align}
It can be shown that $\p_t E \simeq 0$ if~$(S_iH_2^{ij}A_2^p{}_j{}^ks_pS_k)$ 
is Hermitian~\cite{GunGar05}.

\section{Third order systems}
\label{section:Third_order}

Before starting with the generalization to arbitrary order we 
discuss third order systems here. In~\cite{GunGar05} Gundlach 
and Mart\'in-Garc\'ia give different possible definitions of 
hyperbolicity of second order systems. They showed that these 
definitions are equivalent to the existence of a first order 
reduction with the same level of hyperbolicity. We follow a 
similar approach here.

\subsection{Definition of third order systems}


\paragraph*{FT3S systems:} We consider first order in time, 
third order in space (FT3S) systems of the form
\begin{align}
\nonumber
\p_t u &=
(A^u{}_{u})^i \p_i u
+(A^u{}_{v}) v
+(B^u{}_{1u}) u
+s^u,\\
\nonumber
\p_t v &=
(A^v{}_{u})^{ij} \p_i\p_j u
+(A^v{}_{v})^i \p_i v
+(A^v{}_{w}) w
+(B^v{}_{1u})^i\p_i u
+(B^v{}_{2u}) u
+(B^v{}_{1v}) v
+s^v,\\
\nonumber
\p_t w &=
(A^w{}_{u})^{ijk} \p_i\p_j\p_k u
+(A^w{}_{v})^{ij} \p_i\p_j v
+(A^w{}_{w})^{i} \p_i w
+(B^w{}_{1u})^{ij}\p_i\p_j u
+(B^w{}_{2u})^i\p_i u
\\
&\quad
+(B^w{}_{3u}) u+(B^w{}_{1v})^i\p_i v
+(B^w{}_{2v}) v
+(B^w{}_{1w}) w
+s^w,\label{eq:FT3S}
\end{align}
where $s^u$, $s^v$ and $s^w$ are arbitrary source terms that 
do not depend on~$u$,~$v$ or~$w$. In analogy to the second order case 
we define the principal part of that system as
\begin{align}
\nonumber
\p_t u &\simeq
(A^u{}_{u})^i \p_i u
+(A^u{}_{v}) v,\\
\nonumber
\p_t v &\simeq
(A^v{}_{u})^{ij} \p_i\p_j u
+(A^v{}_{v})^i \p_i v
+(A^v{}_{w}) w,\\
\p_t w &\simeq
(A^w{}_{u})^{ijk} \p_i\p_j\p_k u
+(A^w{}_{v})^{ij} \p_i\p_j v
+(A^w{}_{w})^{i} \p_i w,
\nonumber
\end{align}
where as before~$\simeq$ denotes equality up to non principal terms.
As the principal matrix of the system~\eqref{eq:FT3S} we define
\begin{align}
\nonumber
\mathcal A_3^p{}_{kl}{}^{mn} &=\left(
\begin{array}{ccc}
\delta^p_{(k}\delta^{(m}_{l)}(A^u{}_u)^{n)} &
\delta^p_{(k}\delta_{l)}^m (A^u{}_v) & 0 \\
\delta^p_k(A^v{}_u)^{mn} &
\delta^p_k(A^v{}_v)^m & \delta^p_k (A^v{}_w) \\
(A^w{}_u)^{pmn} &
(A^w{}_v)^{pm} & (A^w{}_w)^p
\end{array}
\right)
\end{align}
and the principal symbol is~$P_3^s= S^{ij} A_3^p{}_{ij}{}^{kl} S_{kl} s_p$
where~$S_{ij}=\mbox{diag}(s_is_j,s_i,1)$.

\subsection{Reduction to second order}
\label{sec:2nd_order_reduction}

\paragraph*{Reduction variables:} We are going to define strong 
hyperbolicity of FT3S systems by referring to strong hyperbolicity 
of second order systems. Here we define what we mean by a reduction 
of the FT3S system~\eqref{eq:FT3S} to second order. We introduce 
a vector of reduction variables~$d_a$. The reduction variables 
eventually replace the spatial derivatives of the fields~$u$ in 
the reduced system: $d_a = \p_a u$.

We use lower case letters from the beginning of the Latin 
alphabet as derivative indices without further meaning. In what 
follows their use simply helps to identify indices which belong
to~$d$ which makes it simpler to work with the principal matrix 
of the second order reduction.

\paragraph*{Unmodified evolution equations:} The aim is now to 
include the~$d_a$ as independent variables in a first order in 
time second order in space (FT2S) system. Therefore an evolution 
equation for these variables is needed which must be consistent 
with $d_a = \p_a u$. One gets this equation e.g. by 
taking the spatial derivative of the evolution equation for~$u$:
\begin{align}
\label{eq:FT3S_particular_reduction_only_d}
\p_t d_a &=
(A^u{}_{u})^j \p_a d_j
+(A^u{}_{v}) \p_a v
+(B^u{}_{1u}) \p_a u
+\p_a s^u.
\end{align}
In this equation $\p_a s^u$ does not depend on the variables~$u$,~$v$,~$w$ 
or~$d_a$ and hence can be considered as a given source function.

\paragraph*{Auxiliary constraints:} Obviously the system 
composed of~\eqref{eq:FT3S} 
and~\eqref{eq:FT3S_particular_reduction_only_d} is not second order.
However, one can get rid of the higher order terms by adding linear 
combinations of the following functions and their derivatives to the 
right hand sides
\begin{align}
c_a := \p_a u - d_a,\quad
c_{ia} := \frac12\left(\p_i d_a-\p_a d_i\right),\quad
c_{ija}:= \p_i \p_j d_a - \p_{(i} \p_j d_{a)}.
\label{eq:FT2S_reduction_aux_cnstrs}
\end{align}
These functions vanish when $d_a = \p_a u$
is satisfied. We will show that their evolution system is closed
for the FT2S systems that we consider here.
The functions~$c$ are denoted \emph{auxiliary constraints}.
Furthermore the~$c_{ija}$ can be written as a linear
combination of derivatives of the~$c_{ia}$: $c_{ija} = 2/3 \p_i c_{ja} + 2/3 \p_j c_{ia}$.
Therefore their addition to the right hand sides is already 
covered by the addition of derivatives of the~$c_{ia}$. We do 
not consider the~$c_{ija}$ separately.

\paragraph*{Reduced system:} FT2S systems which are obtained 
in that way have the form
\begin{align}
\label{eq:gen_FT3S_reduction}
\nonumber
\p_t u &=
(A^u{}_{u})^i \p_i u
+(A^u{}_{v}) v
+(B^u{}_{1u}) u
+s^u
+(D^u)^a c_a + (\bar D^u)^{ia} c_{ia}\\
\nonumber \p_t d_a &=
(B^u{}_{1u}) \p_a u
+(A^u{}_{u})^b \p_a d_b
+(A^u{}_{v}) \p_a v
+\p_a s
+(D)_a{}^b c_b + (\bar D)_a{}^{kb} c_{kb},\\
\nonumber
\p_t v &=
(B^v{}_{1u})^i\p_i u
+(A^v{}_{u})^{ia} \p_i d_a
+(A^v{}_{v})^i \p_i v
+(A^v{}_{w}) w
+(B^v{}_{2u}) u
+(B^v{}_{1v}) v
+s^v\\
\nonumber & \quad
+(D^v)^a c_a + (\bar D^v)^{ia} c_{ia},\\
\nonumber
\p_t w &=
(B^w{}_{1u})^{ij}\p_i\p_j u
+(A^w{}_{u})^{ija} \p_i\p_jd_a
+(A^w{}_{v})^{ij} \p_i\p_j v
+(A^w{}_{w})^{i} \p_i w
+(B^w{}_{2u})^i\p_i u\\
\nonumber
& \quad
+(B^w{}_{3u}) u
+(B^w{}_{1v})^i\p_i v
+(B^w{}_{2v}) v
+(B^w{}_{1w}) w
+(D^w)^{ka} \p_kc_a\\
& \quad
+ (\bar D^w)^{kja} \p_kc_{ja}.
\end{align}
We denote the constant matrices $D$ and $\bar D$ the \emph{reduction
parameters}. Since $c_{jb}$ is antisymmetric we can assume without loss of
generality
\begin{align}
\nonumber
(\bar D^u)^{ia} = -(\bar D^u)^{ai},\,\,
\bar D_a{}^{kb} = -\bar D_a{}^{bk},\,\,
(\bar D^v)^{ia} = -(\bar D^v)^{ai},\,\,
(\bar D^w)^{kja} = -(\bar D^w)^{kaj}.
\end{align}

\begin{defi}
We call a first order in time, second order in space system of the form
\eqref{eq:gen_FT3S_reduction} an \emph{FT2S reduction} of the
first order in time, third order in space system~\eqref{eq:FT3S}.
\end{defi}

This definition of a reduction to second order is quite restrictive, one
may think of other definitions that are satisfied by more second order
systems. Indeed one finds that it is too restrictive to be used in 
a definition of symmetric hyperbolicity for FT3S systems. We discuss that aspect
shortly in section \ref{sec:diff_reds}.

\paragraph*{Auxiliary constraint evolution:} For every FT2S 
reduction of~\eqref{eq:FT3S}, provided that the reduction 
constraints are satisfied, one can show that there is a 
relationship between solutions of the two systems.
\begin{lem}
If the system~\eqref{eq:gen_FT3S_reduction}
is an FT2S reduction of~\eqref{eq:FT3S}
and $(u, d_a, v, w)$ is a solution of
\eqref{eq:gen_FT3S_reduction} with vanishing auxiliary constraints
\eqref{eq:FT2S_reduction_aux_cnstrs}
then~$(u,v,w)$ is a solution of the FT3S system~\eqref{eq:FT3S}.
Moreover, if~$(u,v,w)$ is a solution of the FT3S system~\eqref{eq:FT3S} and
the system~\eqref{eq:gen_FT3S_reduction} is an FT2S reduction
of~\eqref{eq:FT3S} then~$(u,\p_a u,v,w)$ is a solution of the FT2S
system~\eqref{eq:gen_FT3S_reduction} with vanishing auxiliary 
constraints~\eqref{eq:FT2S_reduction_aux_cnstrs}.
\end{lem}
\begin{proof}
By inserting the subset~$(u,v,w)$ of the FT2S solution into the FT3S
system one can easily check that these functions satisfy~\eqref{eq:FT3S},
because the auxiliary constraints~\eqref{eq:FT2S_reduction_aux_cnstrs} 
vanish by assumption. Moreover, if~$(u,v,w)$ is a solution 
of~\eqref{eq:FT3S} then one can insert~$(u, \p_a u, v, w)$ into the 
system~\eqref{eq:gen_FT3S_reduction} to see that it is a solution.

The reason for this being that the auxiliary constraint evolution system 
is closed:
\begin{align}
\nonumber
\p_t c_a &= \p_t d_a - \p_a\p_t u\nonumber\\
&=
(A^u{}_{u})^b \p_a c_b
+\left((D)_a{}^b - (D^u)^b \p_a \right) c_b
+\left((\bar D)_a{}^{kb} - (\bar D^u)^{kb} \p_a \right) c_{kb},\nonumber\\
\nonumber
\p_t c_{ia} &= \p_i \p_t d_a - \p_a \p_t d_i\\
\nonumber &=
(D)_a{}^b \p_i c_b + (\bar D)_a{}^{kb} \p_i c_{kb}
-(D)_i{}^b \p_a c_b - (\bar D)_i{}^{kb} \p_a c_{kb},\\
\nonumber
\p_t c_{ija} &= 2/3 \p_j \p_t c_{ia} + 2/3 \p_a \p_t c_{ij}.
\end{align}
It is straightforward to check that $(u, \p_a u, v, w)$ solves
\eqref{eq:FT2S_reduction_aux_cnstrs}.
\end{proof}

\paragraph*{Principal part of the FT2S reduction:} According to 
the definitions given in section \ref{section:Basic_Hyp} the 
principal part of the FT2S reduction~\eqref{eq:gen_FT3S_reduction}
is

\begin{align}
\nonumber
\p_t u &\simeq
(A^u{}_{u})^i \p_i u
+(D^u)^a \p_a u
+(\bar D^u)^{ia} \p_{i} d_a,\\
\nonumber \p_t d_a &\simeq
(B^u{}_{1u}) \p_a u
+(D)_a{}^b \p_b u
+(A^u{}_{u})^b \p_a d_b
+(\bar D)_a{}^{kb} \p_{k}d_b
+(A^u{}_{v}) \p_a v,\\
\nonumber
\p_t v &\simeq
(B^v{}_{1u})^i\p_i u
+(D^v)^a \p_a u
+(A^v{}_{u})^{ia} \p_i d_a
+(\bar D^v)^{ia} \p_{i}d_a
+(A^v{}_{v})^i \p_i v
+(A^v{}_{w}) w,\\
\nonumber
\p_t w &\simeq
(B^w{}_{1u})^{ij}\p_i\p_j u
+(D^w)^{ka} \p_k\p_a u
+(A^w{}_{u})^{ija} \p_i\p_jd_a
+(\bar D^w)^{kja} \p_k\p_{j} d_a\\
\nonumber
&\quad
+(A^w{}_{v})^{ij} \p_i\p_j v
+(A^w{}_{w})^{i} \p_i w
\end{align}
and the principal matrix is
\begin{align}
\mathcal A_2^p{}_i{}^j{}_a{}^b &=
\left(
\begin{array}{cccc}
\delta^p_i\left((A^u{}_u)^j + (D^u)^j\right) &
\delta^p_i (\bar D^u)^{jb} & 0 & 0 \\
\delta^p_i\left((B^u{}_{1u})\delta_a^j + (D)_a{}^j\right) &
\delta^p_i\left((A^u{}_u)^b\delta^j_a + (\bar D)_a{}^{jb}\right) &
(A^u{}_v)\delta^p_i\delta_a^j & 0 \\
\delta^p_i\left((B^v{}_{1u})^{j} + (D^v)^{j}\right) &
\delta^p_i\left((A^v{}_u)^{jb} + (\bar D^v)^{jb}\right) &
\delta^p_i(A^v{}_v)^j & \delta^p_i (A^v{}_w) \\
(B^w{}_{1u})^{pj} + (D^w)^{pj} & (A^w{}_u)^{pjb}+(\bar D^w)^{pjb} &
(A^w{}_v)^{pj} & (A^w{}_w)^p
\end{array}
\right).\nonumber
\end{align}

\subsection{Strong hyperbolicity}
\label{sec:FT3S_strong_hyp}

\paragraph*{Definitions of strong hyperbolicity:} We show that 
the following definitions of third order strong hyperbolicity 
are equivalent
\setcounter{subdefi}{1}
\renewcommand{\thedefi}{\arabic{defi}\alph{subdefi}}
\begin{defi}
\label{def:strong_hyp_FT3S_red}
The FT3S system~\eqref{eq:FT3S} is called FT2S strongly hyperbolic if
there exists an FT2S reduction~\eqref{eq:gen_FT3S_reduction}
which is strongly hyperbolic in the sense
described in section \ref{section:Basic_Hyp}.
\end{defi}
\addtocounter{defi}{-1}
\addtocounter{subdefi}{1}
\begin{defi}
\label{def:strong_hyp_FT3S}
The FT3S system~\eqref{eq:FT3S} is called FT3S strongly hyperbolic if
there exist a constant $M_3>0$ and a family of hermitian matrices
$H_3(s)$ such that
\begin{align}
\label{eq:FT3S:FT3S_str_hyp}
H_3(s) P_3^s = (P_3^s)^\dag H_3(s),\quad
M_3^{-1}\,I  \leq H_3(s) \leq M_3\,I,
\end{align}
where the matrix inequality is understood in the standard sense
\eqref{eq:matrix_inequality}.
\end{defi}
\renewcommand{\thedefi}{\arabic{defi}}
With this one can apply an iterative procedure which reduces strong
hyperbolicity of FT3S systems to strong hyperbolicity of fully 
first order systems. First one reduces the FT3S system to second 
order and after that the resulting FT2S system to a fully first 
order system by applying the work of Gundlach and 
Mart\'in-Garc\'ia~\cite{GunGar05}, FT3S~$\rightarrow$ 
FT2S~$\rightarrow$ FT1S.

A third possible definition of strong 
hyperbolicity employs a pseudo-differential reduction. One finds that 
this definition is very similar to our 
definition~\ref{def:strong_hyp_FT3S}. We discuss the topic 
in section~\ref{app:pseudo-red} for systems or arbitrary order.

\paragraph*{FT2S strong hyperbolicity $\Rightarrow$ FT3S strong hyperbolicity:}
In the proof that definition \ref{def:strong_hyp_FT3S_red}
implies \ref{def:strong_hyp_FT3S} we start with a 2+1 decomposition
of the reduction variable $d_a$. With the orthogonal projector, $q_a^A$, of the given vector $s$
we decompose $d_a = q_a^B d_B + s_a d_s$, where $d_B s^B = 0$.
Partitioning the state vector as~$(u,d_A,d_s,v,w)$ the FT2S 
principal symbol~$P_2^s{}_A{}^B$ is
\begin{align}
\label{eq:FT2S_reduction_PSymbol}
P_2^s{}_A{}^B &=
\left(
\begin{array}{cc}
\tilde X_A{}^B & 0 \\
\tilde Y^B & P_3^s
\end{array}
\right),
\end{align}
where
\begin{align}
\tilde X_A{}^B
&=
\left(
\begin{array}{cc}
\left((A^u{}_u)^j + (D^u)^j\right) s_j &
(\bar D^u)^{jb} s_j q_b^B\\
(D)_a{}^j q_A^a s_j &
(\bar D)_a{}^{jb} q_A^a s_j q_b^B
\end{array}
\right)\nonumber
\end{align}
and
\begin{align}
\tilde Y^B &=
\left(
\begin{array}{cc}
(B^u{}_{1u}) + (D)_a{}^j s^a s_j &
\left((A^u{}_u)^b + (\bar D)_a{}^{jb}s^a s_j\right)q_b^B\\
\left((B^v{}_{1u})^{j} + (D^v)^{j}\right) s_j &
\left((A^v{}_u)^{jb} + (\bar D^v)^{jb}\right)s_j q_b^B\\
\left((B^w{}_{1u})^{pj} + (D^w)^{pj}\right)s_p s_j &
\left((A^w{}_u)^{pjb} + (\bar D^w)^{pjb}\right)s_p s_j q_b^B
\end{array}
\right).\nonumber
\end{align}
There we used that the~$\bar D$ are antisymmetric in the last two indices.
That is, if one contracts both indices with $s$ then the result vanishes.

The assumption that \eqref{eq:FT3S} is FT2S strongly hyperbolic means
that there exist a constant $M_2$ and a family of matrices
$H_2(s)^{AB}$ such that
\begin{align}
\label{eq:FT3S:FT2S_str_hyp_a}
&H_2(s)^{AB} P_2^s{}_B{}^C = (P_2^s{}_B{}^A)^{\dag} H_2(s)^{BC}, \\
\label{eq:FT3S:FT2S_str_hyp_b}\nonumber
&M_2^{-1}\,I^{AB} \leq  H_2(s)^{AB} \leq M_2\,I^{AB},
\end{align}
where $I^{AB}$ is the appropriate identity matrix.

We decompose $H_2(s)^{AB}$ in a way compatible to the decomposition
in \eqref{eq:FT2S_reduction_PSymbol}:
\begin{align}
H_2(s)^{AB} &=
\left(
\begin{array}{cc}
H_{11}(s)^{AB} & H_{12}(s)^{A}\\
H_{12}(s)^{\dag\;B} & H_{22}(s)
\end{array}
\right),\nonumber
\end{align}
and find
\begin{align}
H_2(s)^{AB} P_2^s{}_B{}^C&=
\left(
\begin{array}{cc}
H_{11}(s)^{AB}\tilde X_B{}^C + H_{12}(s)^{A}\tilde Y{}^C &
H_{12}(s)^{A} P_3^s \\
H_{12}(s)^{\dag\;B}\tilde X_B{}^C + H_{22}(s)\tilde Y^C &
H_{22}(s) P_3^s
\end{array}
\right).
\end{align}
Looking at the lower right block of this expression
equation \eqref{eq:FT3S:FT2S_str_hyp_a} implies 
$H_{22}(s) P_3^s = (P_3^s)^\dag H_{22}(s)$.
Furthermore we have obviously $H_{22}(s) = H_{22}(s)^\dag$ and
\begin{align}
M_2^{-1} v^\dag v &\leq
v^\dag H_{22}(s) v
=
\left(0, v^\dag\right) H_{2}(s)^{AB} \left(0, v^\dag\right)^\dag
\leq M_2 v^\dag v\quad \forall v,\nonumber
\end{align}
because \eqref{eq:FT3S:FT2S_str_hyp_b}
is satisfied by assumption.

Hence, the matrix $H_3(s):=H_{22}(s)$ satisfies \eqref{eq:FT3S:FT3S_str_hyp}
and FT3S strong hyperbolicity of \eqref{eq:FT3S} is shown.

\paragraph*{FT3S strong hyperbolicity $\Rightarrow$ FT2S strong hyperbolicity:}
For the reverse direction we need to choose the reduction parameters
appropriately. One can check easily that the first row and column of
\eqref{eq:FT2S_reduction_PSymbol} vanish with the choice
\begin{align}
(D^u)^j &= -(A^u{}_u)^j,&
(D)_a{}^j &= -(B^u{}_{1u})\delta_a^j,&
(D^v)^{j} &= -(B^v{}_{1u})^{j},\nonumber\\
(D^w)^{pj} &= -(B^w{}_{1u})^{pj},&
(\bar D^u)^{jb} &= 0. &
\label{eq:FT3S_partial_choice}
\end{align}
We call~\eqref{eq:FT3S_partial_choice} the \emph{partial choice} 
of reduction parameters. Under the partial choice~$P_2^s{}_A{}^B$ 
has the following lower block triangular form,
\begin{align}
P_2^s{}_A{}^B &=
\left(
\begin{array}{ccc}
0 & 0 & 0 \\
0 & X_A^B & 0 \\
0 & Y^B & P_3^s \\
\end{array}
\right),\nonumber
\end{align}
where
\begin{align}
\label{eq:FT3S_XAB}
X_A^B &= (\bar D)_a{}^{jb} q_A^a s_j q_b^B, \nonumber\\
Y^B &=
\left(
\begin{array}{c}
\left((A^u{}_u)^b + (\bar D)_a{}^{jb}s^a s_j\right)q_b^B \\
\left((A^v{}_u)^{jb} + (\bar D^v)^{jb}\right)s_j q_b^B \\
\left((A^w{}_u)^{pjb} + (\bar D^w)^{pjb}\right)s_p s_j q_b^B
\end{array}
\right).
\end{align}

As mentioned in section \ref{section:Basic_Hyp},
definition \ref{def:strong_hyp_FT3S_red} is equivalent to
the existance of a constant $K_2$ and a family of
matrices $T_2(s)_A{}^B$ with
$K_2^{-1} \leq \| T_2(s)_A{}^B \| \leq K_2$
such that $T_2(s)^{-1}{}_A{}^B P_2^s{}_B{}^C\,T_2(s){}_C{}^D$ is real and diagonal.
Here we show this property instead of the original definition.

Following~\cite{GunGar05} we choose the reduction parameters
such that $X_A^B$ is diagonalizable:
\begin{align}
\label{eq:barD_epsilon_choice}
(\bar D)_a{}^{jb} &= i\lambda \varepsilon_a{}^{jb}
\end{align}
with $\lambda\in\mathbb R$ and~$\varepsilon_a{}^{jb}$ the Levi-Civita symbol. The 
eigenvalues of~$X_A^B$ become $\pm\lambda$. They are independent 
of $s$ and the eigenvalues of $P_2^s{}_A{}^B$ are the union of the
eigenvalues of $P_3^s$ and $\pm\lambda$.

Using that $P_3^s$ is bounded, because it is a sum of
products of bounded matrices:
\begin{align}
\|P_3^s\| &= \|S^{ij}\mathcal A^{p}{}_{ij}{}^{kl}S_{kl}s_p\|
\leq \|S_{ij}\|\, \|\mathcal A^{p}{}_{ij}{}^{kl}\|\, \|S_{kl}\|\, \|s_p\|,
\end{align}
we choose $\lambda$ larger than all eigenvalues of $P_3^s$.
Together with the assumption that \eqref{eq:FT3S}
is FT3S strongly hyperbolic, i.e. that $P_3^s$ is
diagonalizable, this choice of $\lambda$ makes $P_2^s{}_A{}^B$
diagonalizable as well.

The corresponding similarity transformation can be constructed
from the eigenvectors of $P_2^s{}_A{}^B$. One finds that
given an eigenvector, $v$, of $P_3^s$ with
eigenvalue $\alpha$ and an eigenvector, $w_B$, of $X_A^B$ then
\begin{align}
\nonumber
P_2^s{}_A{}^B
\left(
\begin{array}{c}
0\\
v
\end{array}
\right)
&=
\alpha
\left(
\begin{array}{c}
0\\
v
\end{array}
\right),&
P_2^s{}_A{}^B
\left(
\begin{array}{c}
w_B\\
w
\end{array}
\right)
&=
\lambda
\left(
\begin{array}{c}
w_B\\
w
\end{array}
\right),
\end{align}
where we used $w := (\lambda-P_3^s)^{-1}Y^B w_B$, which exists,
because $\lambda$ does not coincide with an eigenvalue of $P_3^s$.

Now, a matrix which makes
$T_2(s)^{-1}{}_A{}^B P_2^s{}_B{}^C T_2(s){}_C{}^D$ diagonal (and real) is
\begin{align}
T_2(s){}_A{}^B &=
\left(
\begin{array}{cc}
W_A{}^B & 0 \\
(\lambda-P_3^s)^{-1} Y^A W_A{}^B & T_3(s)
\end{array}
\right),\nonumber
\end{align}
where $T_3(s)$ and $W_A{}^B$ diagonalize $P_3^s$ and
$X_A{}^{B}$ respectively. The inverse of
$T_2(s){}_A{}^B$ is
\begin{align}
\nonumber
T_2(s)^{-1}{}{}_A{}^B &=
\left(
\begin{array}{cc}
W^{-1}{}_A{}^B & 0 \\
-T_3(s)^{-1}(\lambda-P_3^s)^{-1} Y^B & T_3(s)^{-1}
\end{array}
\right).
\end{align}
Both, $T_2(s){}_A{}^B$ and its inverse are bounded, because
on the one hand $T_3(s)$ and $T_3(s)^{-1}$ are bounded by
the assumption \eqref{eq:FT3S:FT3S_str_hyp} and we have chosen
$\lambda$ such that $(\lambda-P_3^s)^{-1}$ is bounded as well.

Hence, we get that there exists a constant $K_2>0$ such that
$K_2^{-1} \leq \|T_2(s){}_A{}^B\| \leq K_2$, which shows that FT3S 
strong hyperbolicity implies FT2S strong hyperbolicity.$\,\Box$

\subsection{Why two different reductions?}
\label{sec:diff_reds}

\paragraph*{Failure of the iterative procedure for symmetric 
hyperbolicity:} 
Symmetric hyperbolicity relies fundamentally on conserved 
quantities (we will 
discuss the details of FT3S conservation equations in 
section~\ref{sec:FT2S_red_sym_hyp}). Hence, in order to deal 
with symmetric hyperbolicity for the second order reductions, 
which were used to handle strong hyperbolicity, we need to 
construct a reduction with a conserved quantity that is associated 
to the given FT3S symmetrizer. However, one finds that there 
are FT3S systems with a conserved energy for which no 
FT2S reduction with the same conserved quantity exists. We derive 
such a counterexample explicitly in the notebook 
{\tt counter\_example\_3rd\_order\_sym\_hyp.nb} which is available 
online~\footnote{{\tt http://www.tpi.uni-jena.de/\~{}hild/FTNS.tgz}},
but for brevity do not give details here.\\ \\

\paragraph*{Discussion:}
This situation differs from the case of reductions of FT2S 
systems to first order. There every FT2S symmetrizer implies 
an FT1S conserved energy. In~\cite{GunGar05} this was the basis 
of the proof that for every symmetric hyperbolic FT2S system 
there exists a symmetric hyperbolic first order reduction.
Thus, we cannot use the iterative procedure to prove existence 
of symmetric hyperbolic lower order reductions. In order to avoid this 
problem we employ a direct reduction to first order (described 
in section \ref{sec:FT3S_1st_order_reduction}) and construct a 
conserved quantity for the first order system.\\ \\

\paragraph*{Why not always use the direct reduction?} Conversely, 
one may also think of using the direct first order reduction to
show statements about strong hyperbolicity. There the problem is 
that the proofs rely on the choice of reduction 
parameters~\eqref{eq:barD_epsilon_choice}. For direct first order
reductions the structure of reduction parameters changes completely,
and  we did not find a choice that shows existence
of a strongly hyperbolic direct first order reduction. So we use one class 
of reductions for proofs about strong hyperbolicity, namely 
reductions from FT3S to FT2S, and another class for proofs on 
symmetric hyperbolicity, namely reductions from FT3S to FT1S.

\subsection{Direct reduction to first order}
\label{sec:FT3S_1st_order_reduction}

\paragraph*{Reduction variables:} In analogy to the construction 
of FT2S reductions of the FT3S system~\eqref{eq:FT3S} we now 
define \emph{direct first order reductions} of~\eqref{eq:FT3S}. We 
also use the terminology \emph{direct FT1S reduction}. 
We define reduction variables
$d^u_i = \p_i u, 
d^u_{ij} = \p_{(i} d^u_{j)}$ and $ 
d^v_i = \p_i v$.
The equations of motion which one derives from these definitions 
are
\begin{align}
\nonumber
\p_t d^u_i &=(A^u{}_{u})^j \p_i \p_j u
+(A^u{}_{v}) \p_i v+(B^u{}_{1u}) \p_i u + \p_i s^u,\\
\nonumber
\p_t d^u_{ij} &=(A^u{}_{u})^k \p_i \p_j \p_k u
+(A^u{}_{v}) \p_i \p_j v+(B^u{}_{1u}) \p_i \p_j u
+\p_i \p_j s^u,\\
\nonumber
\p_t d^v_i &= 
(A^v{}_{u})^{jk} \p_i\p_j\p_k u
+(A^v{}_{v})^j \p_i\p_j v
+(A^v{}_{w}) \p_i w
+(B^v{}_{1u})^j\p_i\p_j u
+(B^v{}_{2u}) \p_i u\\
&\quad
+(B^v{}_{1v}) \p_i v+\p_i s^v.
\label{eq:1st_order_red_var_eom}
\end{align}

\paragraph*{Auxiliary constraints:} They are subject to the first order 
auxiliary constraints
\begin{align}
\label{eq:aux_FT3-1S}
\nonumber
c^u_i &= \p_i u - d^u_i, &
\bar c^u_{ij} &= \p_i d^u_j - \p_{(i} d^u_{j)}, &
c^u_{ij} &= \p_{(i} d^u_{j)} - d^u_{ij},\\
\bar c^u_{ijk} &= \p_i d^u_{jk} - \p_{(i} d^u_{jk)}, &
c^v_i &= \p_i v - d^v_i, & 
\bar c^v_{ij} &= \p_i d^v_j - \p_{(i} d^v_{j)}.
\end{align}
We call a first order system which is composed of equations
\eqref{eq:FT3S} and~\eqref{eq:1st_order_red_var_eom} with
additions of linear combinations of the auxiliary constraints
\eqref{eq:aux_FT3-1S} and their derivatives to the right hand sides
a \emph{direct first order reduction} of the FT3S system~\eqref{eq:FT3S}.

\paragraph*{Reduction:} Note that we allow additions of 
derivatives of the auxiliary constraints, but it is not possible 
to add arbitrary derivatives, because the final system must be 
first order. The constraint additions are used to cancel the 
higher order terms in~\eqref{eq:FT3S} 
and~\eqref{eq:1st_order_red_var_eom}.

As in section \ref{sec:2nd_order_reduction} one can show that there 
is a one-to-one relation between solutions of~\eqref{eq:FT3S} and 
the solutions of first order reductions which satisfy the 
auxiliary constraints. The reason is again that the auxiliary 
constraint evolution system in the first order reduction is 
closed. We show this step for arbitrary spatial derivative order 
in section~\ref{section:Reduction}.

The principal part of a first 
order reduction of~\eqref{eq:FT3S} has the form
\begin{align}
\p_t u & \simeq
(D^u{}_u)^{k} c^u_k
+ (D^u{}_u)^{kl} c^u_{kl}
+ (D^u{}_v)^{k} c^v_k
+ (\bar D^u{}_u)^{kl} \bar c^u_{kl}
+ (\bar D^u{}_u)^{klm} \bar c^u_{klm}\nonumber\\
& \quad + (\bar D^u{}_v)^{kl} \bar c^v_{kl}, \nonumber\\
\p_t d^u_{i} & \simeq
(D^u{}_u)_{i}{}^{k} c^u_k
+ (D^u{}_u)_{i}{}^{kl} c^u_{kl}
+ (D^u{}_v)_{i}{}^{k} c^v_k
+ (\bar D^u{}_u)_{i}{}^{kl} \bar c^u_{kl}
+ (\bar D^u{}_u)_{i}{}^{klm} \bar c^u_{klm}\nonumber\\
& \quad + (\bar D^u{}_v)_{i}{}^{kl} \bar c^v_{kl},\nonumber\\
\p_t v & \simeq (D^v{}_u)^{k} c^u_k+ (D^v{}_u)^{kl} c^u_{kl}
+ (D^v{}_v)^{k} c^v_k + (\bar D^v{}_u)^{kl} \bar c^u_{kl}
+ (\bar D^v{}_u)^{klm} \bar c^u_{klm}\nonumber\\
& \quad + (\bar D^v{}_v)^{kl} \bar c^v_{kl},\nonumber\\
\p_t d^u_{ij} &\simeq
(A^u{}_{u})^k \p_{(i} d^u_{j)k}
+(A^u{}_{v}) \p_{(i} d^v_{j)}
+ (D^u{}_u)_{ij}{}^{k} c^u_k
+ (D^u{}_u)_{ij}{}^{kl} c^u_{kl}
+ (D^u{}_v)_{ij}{}^{k} c^v_k\nonumber\\
& \quad
+ (\bar D^u{}_u)_{ij}{}^{kl} \bar c^u_{kl}
+ (\bar D^u{}_u)_{ij}{}^{klm} \bar c^u_{klm}
+ (\bar D^u{}_v)_{ij}{}^{kl} \bar c^v_{kl},\nonumber\\
\p_t d^v_i & \simeq (A^v{}_{u})^{jk} \p_i d^u_{jk}
+(A^v{}_{v})^j \p_i d^v_j
+(A^v{}_{w}) \p_i w
+ (D^v{}_u)_{i}{}^{k} c^u_k
+ (D^v{}_u)_{i}{}^{kl} c^u_{kl}\nonumber\\
& \quad
+ (D^v{}_v)_{i}{}^{k} c^v_k
+ (\bar D^v{}_u)_{i}{}^{kl} \bar c^u_{kl}
+ (\bar D^v{}_u)_{i}{}^{klm} \bar c^u_{klm}
+ (\bar D^v{}_v)_{i}{}^{kl} \bar c^v_{kl},\nonumber\\
\nonumber \p_t w & \simeq
(A^w{}_{u})^{ijk} \p_i d^u_{jk}
+(A^w{}_{v})^{ij} \p_i d^v_j
+(A^w{}_{w})^{i} \p_i w
+ (D^w{}_u)^{k} c^u_k
+ (D^w{}_u)^{kl} c^u_{kl}\\
& \quad
+ (D^w{}_v)^{k} c^v_k
+ (\bar D^w{}_u)^{kl} \bar c^u_{kl}
+ (\bar D^w{}_u)^{klm} \bar c^u_{klm}
+ (\bar D^w{}_v)^{kl} \bar c^v_{kl},
\label{eq:FT3S_1st_order_red}
\end{align}
where the constant matrices $(D^X{}_Y)$ and $(\bar D^X{}_Y)$ ($X,Y = u,v,w$)
are the reduction parameters.

Since the reduction parameters are contracted
with auxiliary constraints and the symmetric part of the $\bar c$ vanishes
we assume without loss of generality that the~$\bar D$ symmetrized in the 
upper indices vanish:
\begin{align}
\label{eq:FT3-1S:Dbar_sym}
\nonumber
(\bar D^X{}_Y)^{(kl)} &= 0, &
(\bar D^X{}_Y)^{(klm)} &= 0,& 
(\bar D^X{}_Y)_{i}{}^{(kl)} &= 0, \\
(\bar D^X{}_Y)_{i}{}^{(klm)} &= 0, &
(\bar D^X{}_Y)_{ij}{}^{(kl)} &= 0, &
(\bar D^X{}_Y)_{ij}{}^{(klm)} &= 0.
\end{align}
Moreover, since~$d_{ij}^u=d_{(ij)}^u$ the reduction variables
satisfy
\begin{align}
\nonumber
(D^u{}_u)^{kl} &= (D^u{}_u)^{(kl)}, &
(D^u{}_u)_{i}{}^{kl} &= (D^u{}_u)_{i}{}^{(kl)}, &
(D^v{}_u)^{kl} &= (D^v{}_u)^{(kl)}, \\
\nonumber
(D^v{}_u)_{i}{}^{kl} &= (D^v{}_u)_{i}{}^{(kl)}, &
(D^w{}_u)^{kl} &= (D^w{}_u)^{(kl)}, &
(\bar D^u{}_u)^{klm} &= (\bar D^u{}_u)^{k(lm)}, \\
\nonumber
(\bar D^u{}_u)_{i}{}^{klm} &= (\bar D^u{}_u)_{i}{}^{k(lm)}, &
(\bar D^v{}_u)^{klm} &= (\bar D^v{}_u)^{k(lm)}, &
(\bar D^v{}_u)_{i}{}^{klm} &= (\bar D^v{}_u)_{i}{}^{k(lm)}, \\
\nonumber
(\bar D^w{}_u)^{klm} &= (\bar D^w{}_u)^{k(lm)}, &
(\bar D^u{}_u)_{ij}{}^{klm} &= (\bar D^u{}_u)_{(ij)}{}^{k(lm)}, &
(\bar D^u{}_u)_{ij}{}^{kl} &= (\bar D^u{}_u)_{(ij)}{}^{kl}, \\
\nonumber
(D^u{}_u)_{ij}{}^{kl} &= (D^u{}_u)_{(ij)}{}^{(kl)}, &
(D^u{}_u)_{ij}{}^{k} &= (D^u{}_u)_{(ij)}{}^{k}, &
(D^u{}_v)_{ij}{}^{k} &= (D^u{}_v)_{(ij)}{}^{k}.
\end{align}
In a representation with the state vector
$
\left(
u, d_i^u, v, d_{ij}^u, d_{i}^v, w
\right)
$
the principal matrix of the system~\eqref{eq:FT3S_1st_order_red} is
\begin{align}
\label{eq:FT3S_first_order_PMat}
&\mathcal A_1^{p}{}_{ij}{}^{kl} =\\
\nonumber &\left(
\begin{array}{cccccc}
(D^u{}_u)^p &
(\tilde D^u{}_u)^{pk} &
(D^u{}_v)^p &
(\bar D^u{}_u)^{pkl} &
(\bar D^u{}_v)^{pk} &
0 \\
(D^u{}_u)_{i}{}^p &
(\tilde D^u{}_u)_{i}{}^{pk} &
(D^u{}_v)_{i}{}^p &
(\bar D^u{}_u)_{i}{}^{pkl} &
(\bar D^u{}_v)_{i}{}^{pk} &
0 \\
(D^v{}_u)^p &
(\tilde D^v{}_u)^{pk} &
(D^v{}_v)^p &
(\bar D^v{}_u)^{pkl} &
(\bar D^v{}_v)^{pk} &
0 \\
(D^u{}_u)_{ij}{}^p &
(\tilde D^u{}_u)_{ij}{}^{pk} &
(D^u{}_v)_{ij}{}^p &
(A^u{}_u)^{(k}\delta^{l)}_{(i}\delta^p_{j)}+(\bar D^u{}_u)_{ij}{}^{pkl} &
A^u{}_v\delta^p_{(i}\delta^k_{j)}+(\bar D^u{}_v)_{ij}{}^{pk} &
0 \\
(D^v{}_u)_{i}{}^p &
(\tilde D^v{}_u)_{i}{}^{pk} &
(D^v{}_v)_{i}{}^p &
(A^v{}_u)^{kl}\delta^p_i+(\bar D^v{}_u)_{i}{}^{pkl} &
(A^v{}_v)^{k}\delta^p_i+(\bar D^v{}_v)_{i}{}^{pk} &
(A^v{}_w)\delta^{p}_i \\
(D^w{}_u)^p &
(\tilde D^w{}_u)^{pk} &
(D^w{}_v)^p &
(A^w{}_u)^{pkl} + (\bar D^w{}_u)^{pkl} &
(A^w{}_v)^{pk} + (\bar D^w{}_v)^{pk} &
(A^w{}_w)^{p}
\end{array}
\right),
\end{align}
where $(\tilde D^X{}_Y)_{*}{}^{\star} := (\bar D^X{}_Y)_{*}{}^{\star} 
+ (D^X{}_Y)_{*}{}^{\star}$.

\subsection{Symmetric hyperbolicity}
\label{sec:FT2S_red_sym_hyp}

\paragraph*{Definitions of symmetric hyperbolicity:} Now we 
show that the following definitions of third order symmetric 
hyperbolicity are equivalent
\setcounter{subdefi}{1}
\renewcommand{\thedefi}{\arabic{defi}\alph{subdefi}}
\begin{defi}
\label{def:sym_hyp_FT3S_red}
The FT3S system~\eqref{eq:FT3S} is called first order
symmetric hyperbolic if
there exists a first order reduction which is symmetric 
hyperbolic in the usual first order sense~\cite{GusKreOli95}.
That is, there exists a choice of reduction parameters
and a Hermitian matrix $H_1^{ij\,kl} = H_1^{(ij)\,(kl)}$ which is
positive definite in the space of symmetric tensors such 
that~$H_1^{ij\,kl}\mathcal A_1^{p}{}_{kl}{}^{mn}$ is Hermitian for 
all $p$.
\end{defi}
\addtocounter{defi}{-1}
\addtocounter{subdefi}{1}
The matrix $H_1^{ij\,kl}$ is symmetric in $(i,j)$ and $(k,l)$,
because we defined the reduction variable $d^u_{ij}$ symmetric.
\begin{defi}
\label{def:FT3S_sym_hyp_direct}
The FT3S system~\eqref{eq:FT3S} is called FT3S symmetric hyperbolic if
there exists a Hermitian matrix $H_3^{ij\,kl}=H_3^{(ij)\,(kl)}$ which
is positive definite in the space of symmetric tensors such 
that~$S_{ij} H_3^{ij\,kl} \mathcal A_3^{p}{}_{kl}{}^{mn} s_p S_{mn}$
is Hermitian for every spatial vector $s$.
\end{defi}
\renewcommand{\thedefi}{\arabic{defi}}
As before we denote a positive definite Hermitian matrix
$H_3^{ij\,kl} = H_3^{(ij)\,(kl)}$ which makes the above product
Hermitian a \emph{symmetrizer}.
If~$H_3^{ij\,kl}$ makes the product
Hermitian, but is not necessarily positive
definite then we call it a \emph{candidate symmetrizer}.
It is straightforward to check that given an FT3S 
symmetrizer~$H_3^{ij\,kl}$ the energy
\begin{align}
\nonumber
E &= \int d^3x\, \epsilon = \int d^3x\, u_{ij}^\dag H_3^{ij\,kl} u_{ij}
\end{align}
is conserved up to non principal terms, i.e. $\p_t E \simeq 0$.

\paragraph*{Def. \ref{def:sym_hyp_FT3S_red} $\Rightarrow$ 
Def. \ref{def:FT3S_sym_hyp_direct}:}
Given an FT3S system which satisfies definition \ref{def:sym_hyp_FT3S_red}
there exist, according to the usual definition of symmetric hyperbolicity for
first order systems~\cite{GusKreOli95}, reduction parameters, $D$ and 
$\bar D$, and a matrix $H_1^{ij\,kl} = H_1^{(ij)\,(kl)}$ such that the 
product $H_1^{ij\,kl} \mathcal A_1^{p}{}_{kl}{}^{mn}$
is Hermitian for every $p$. Moreover the matrix $H_1^{ij\,kl}$ is
positive definite in the space of symmetric tensors.

Using the state vector as above both $H_1^{ij\,kl}$ and 
$\mathcal A_1^{p}{}_{kl}{}^{mn}$ are decomposed into $6\times 6$ blocks. 
By grouping the variables as~$~\left(
u, d_{i}^u, v \,\, | \,\,
d_{ij}^u, d_{i}^v, w
\right)$ we identify four $3\times 3$ sub matrices in
$H_1$ and $\mathcal A_1^p$, where $H_1$ has the form
\begin{align}
\label{eq:2x2_decomp_H1}
H_1^{ij\,kl} &=
\left(
\begin{array}{cc}
H_{11}^{i\,k} & H_{12}^{i\,kl} \\
H_{21}^{ij\,k} & H_{22}^{ij\,kl}
\end{array}
\right).
\end{align}

We are now interested in the lower right~$3\times 3$ sub matrix
of the product of $H_1$ and $\mathcal A_1^p$.
It turns out that this sub matrix
contains the FT3S conservation equation, i.e. the condition 
that~$S_{ij} H_3^{ij\,kl} \mathcal A_3^{p}{}_{kl}{}^{mn} s_p S_{mn}$
is Hermitian.

The lower right block of the product
$H_1^{ij\,kl} \mathcal A_1^{p}{}_{kl}{}^{mn}$ is
\begin{align}
\label{eq:lower_3x3_Ft1S}
&H_{21}^{ij\,k}
\left(
\begin{array}{ccc}
(\bar D^u{}_u)^{pmn} &
(\bar D^u{}_v)^{pm} &
0 \\
(\bar D^u{}_u)_{k}{}^{pmn} &
(\bar D^u{}_v)_{k}{}^{pm} &
0 \\
(\bar D^v{}_u)^{pmn} &
(\bar D^v{}_v)^{pm} &
0
\end{array}
\right)\\
\nonumber
&+H_{22}^{ij\,kl}
\left(
\begin{array}{ccc}
(A^u{}_u)^{(m}\delta^{n)}_{(k}\delta^p_{l)}+(\bar D^u{}_u)_{kl}{}^{pmn} &
A^u{}_v\delta^p_{(k}\delta^m_{l)}+(\bar D^u{}_v)_{kl}{}^{pm} &
0 \\
(A^v{}_u)^{mn}\delta^p_k+(\bar D^v{}_u)_{k}{}^{pmn} &
(A^v{}_v)^{m}\delta^p_k+(\bar D^v{}_v)_{k}{}^{pm} &
(A^v{}_w)\delta^{p}_k \\
(A^w{}_u)^{pmn} + (\bar D^w{}_u)^{pmn} &
(A^w{}_v)^{pm} + (\bar D^w{}_v)^{pm} &
(A^w{}_w)^{p}
\end{array}
\right).
\end{align}
By assumption this matrix is Hermitian, because 
it is a quadratic
subblock on the diagonal of the Hermitian matrix
$H_1^{ij\,kl} \mathcal A_1^{p}{}_{kl}{}^{mn}$.

Furthermore, when we contract the index $p$ in~\eqref{eq:lower_3x3_Ft1S}
with an arbitrary
spatial vector $s_p$ and the full matrix from the left and right with
$S_{ij}$ and $S_{mn} = \mbox{diag}(s_ms_n,s_m,1)$ respectively then the 
result is still Hermitian, because $S_{ij}$ and $S_{mn}$ are Hermitian.

Using the fact that the symmetrization of the reduction parameters
$\bar D$ in all upper indices vanishes according to
\eqref{eq:FT3-1S:Dbar_sym} it follows that all terms in
\eqref{eq:lower_3x3_Ft1S} that
contain reduction parameters vanish after the contractions with
$s_p$, $S_{ij}$ and $S_{mn}$. The remaining terms are
\begin{align}
\nonumber
S_{ij}&H_{22}^{ij\,kl}
\left(
\begin{array}{ccc}
(A^u{}_u)^{(k}\delta^{l)}_{(m}\delta^p_{n)} &
A^u{}_v\delta^p_{(k}\delta^m_{l)} &
0 \\
(A^v{}_u)^{mn}\delta^p_k &
(A^v{}_v)^{m}\delta^p_k &
(A^v{}_w)\delta^{p}_k \\
(A^w{}_u)^{pmn} &
(A^w{}_v)^{pm} &
(A^w{}_w)^{p}
\end{array}
\right)s_p S_{mn}
=S_{ij}H_{22}^{ij\,kl}\mathcal A_3^p{}_{kl}{}^{mn} s_p S_{mn}.
\end{align}
It is clear that $H_{22}^{ij\,kl} = H_{22}^{(ij)\,(kl)}$,
and since it is a principal minor of the positive definite
$H_{1}^{ij\,kl}$ it is positive definite as well.
With the identification~$H_{3}^{ij\,kl} = H_{22}^{ij\,kl}$
this shows that the FT3S system is symmetric hyperbolic in the sense
of definition \ref{def:FT3S_sym_hyp_direct}.

\paragraph*{Def. \ref{def:FT3S_sym_hyp_direct} $\Rightarrow$ 
Def. \ref{def:sym_hyp_FT3S_red}:} Given a 
matrix~$H_3^{ij\,kl} = H_3^{(ij)\,(kl)}$ which 
satisfies~$S_{ij} H_3^{ij\,kl} \mathcal A_3^{p}{}_{kl}{}^{mn} s_p S_{mn}$ Hermitian, 
we now construct a symmetric hyperbolic first order reduction of~\eqref{eq:FT3S}. 
At first it is convenient to make a partial choice of the 
reduction parameters such that the first three rows and columns 
of~\eqref{eq:FT3S_first_order_PMat} vanish. This is achieved by
choosing all reduction parameters~$D^X{}_Y=0$ ($X,Y = u,v,w$)
and in addition
\begin{align}
\label{eq:FT3S_1st_order_partial_choice}
\nonumber
(\bar D^u{}_u)^{pk} &= 0, &
(\bar D^u{}_u)_{i}{}^{pk} &= 0, &
(\bar D^v{}_u)^{pk} &= 0, &
(\bar D^u{}_u)_{ij}{}^{pk} &= 0, \\
\nonumber
(\bar D^v{}_u)_{i}{}^{pk} &= 0, &
(\bar D^w{}_u)^{pk} &= 0, &
(\bar D^u{}_u)^{pkl} &= 0, &
(\bar D^u{}_u)_{i}{}^{pkl} &= 0, \\
(\bar D^v{}_u)^{pkl} &= 0, &
(\bar D^u{}_v)^{pk} &= 0, &
(\bar D^u{}_v)_{i}{}^{pk} &= 0, &
(\bar D^v{}_v)^{pk} &= 0.
\end{align}
The next step is to make the ansatz
\begin{align}
\label{eq:FT3S_H1_Ansatz}
H_1^{ij\,kl} &=
\left(
\begin{array}{cc}
\Gamma^{ik} & 0 \\
0 & H_3^{ij\,kl}
\end{array}
\right)
=
\left(
\begin{array}{cccc}
\mathbf{1} & 0 & 0 & 0 \\
0 & \gamma^{ik} & 0 & 0 \\
0 & 0 & \mathbf{1} & 0 \\
0 & 0 & 0 & H_3^{ij\,kl}
\end{array}
\right),
\end{align}
where the $2\times 2$ decomposition here is to be 
understood in the same sense as in~\eqref{eq:2x2_decomp_H1}. 
Obviously this matrix is positive definite when~$H_3^{ij\,kl}$ is.
Hence, what needs to be shown with this ansatz is that
the remaining reduction parameters can be chosen such 
that~$S_{ij} H_3^{ij\,kl} \mathcal A_3^{p}{}_{kl}{}^{mn} s_p S_{mn}$ is 
Hermitian for all $p$.

We define
\begin{align}
\label{eq:def_FT3S_J}
J^{p\,ij\,mn} &:=
H_3^{ij\,kl}
\left(
\begin{array}{ccc}
(\bar D^u{}_u)_{kl}{}^{pmn} &
(\bar D^u{}_v)_{kl}{}^{pm} &
0 \\
(\bar D^v{}_u)_{k}{}^{pmn} &
(\bar D^v{}_v)_{k}{}^{pm} &
0 \\
(\bar D^w{}_u)^{pmn} &
(\bar D^w{}_v)^{pm} &
0
\end{array}
\right)
\end{align}
and~$T^{p\,ij\,nm} := H_3^{ij\,kl} \mathcal A_3^p{}_{kl}{}^{mn}$.

Since the form of $\mathcal A_1^{p}{}_{ij}{}^{kl}$ and $H_1^{ij\,kl}$
has been simplified by taking the partial 
choice~\eqref{eq:FT3S_1st_order_partial_choice} and the 
ansatz~\eqref{eq:FT3S_H1_Ansatz} respectively we only need to show that 
there exist reduction parameters such that the matrix 
\begin{align}
\label{eq:FT3S_1st_order_herm_cond}
J^{p\,ij\,mn} + T^{p\,ij\,nm}
\end{align}
is Hermitian for all $p$. In this equation $T^{p\,ij\,nm}$ is fixed
because we assume an FT3S system with given symmetrizer.

The condition that $H_3^{ij\,kl}$ is a candidate symmetrizer is
equivalent to
$T^{(p\,ij\,nm)} = T^{\dag\,(p\,ij\,nm)}$,
because for all tensors $X^{p\,ij\,kl}$ the equivalence:~$X^{(p\,ij\,kl)} = 0
\Leftrightarrow s_ps_{i}s_{j}X^{p\,ij\,kl}s_{k}s_{l}=0\;\forall s$ holds.

Now we need to find an appropriate $J^{p\,ij\,nm}$.
In order to be able to solve~\eqref{eq:def_FT3S_J}
for the reduction parameters it needs to satisfy certain
symmetries:
\begin{align}
\label{eq:FT3S_J_symmetries}
J^{(p|\,ij\,|kl)} &= 0, &
J^{p\,ij\,kl} &= J^{p\,(ij)\,(kl)}.
\end{align}
Note that $J^{(p|\,ij\,|kl)}=0$ implies that the
last column of $J^{p\,ij\,kl}$ vanishes.

One can prove the existence of a $J^{p\,ij\,kl}$ which
satisfies~\eqref{eq:FT3S_J_symmetries} 
and makes~\eqref{eq:FT3S_1st_order_herm_cond} Hermitian
by construction. With the definition
$V^{p\,ij\,kl} := T^{p\,ij\,kl} - T^{\dag p\,kl\,ij}$
the condition that~\eqref{eq:FT3S_1st_order_herm_cond}
is Hermitian becomes
\begin{align}
\label{eq:FT3S_J_eq}
J^{p\,ij\,kl} - J^{\dag p\,kl\,ij} &= -V^{p\,ij\,kl}.
\end{align}

In the {\tt Mathematica} notebook {\tt flux\_construction.nb} 
accompanying the paper~\footnote{\tt http://www.tpi.uni-jena.de/\~{}hild/FTNS.tgz}
we show that using the ansatz,
\begin{align}
J^{p\,ij\,kl} &= \sum_{\pi\in S_5} x_\pi V^{\pi(p)\,\pi(i)\pi(j)\,\pi(k)\pi(l)}\nonumber
\end{align}
the system~\eqref{eq:FT3S_J_symmetries},\eqref{eq:FT3S_J_eq}
becomes a linear system on the $x_\pi$, which can be
solved if $V^{(p\,ij\,kl)}=0$.

The latter condition is satisfied by assumption. Hence, 
multiplication of the resulting $J^{p\,ij\,kl}$ from the left by $H_3^{-1}$
(which exists, because $H_3$ is positive definite)
shows that there exists a first order reduction which is symmetric hyperbolic
and has the symmetrizer~\eqref{eq:FT3S_H1_Ansatz}. $\Box$

\section{Higher order systems}
\label{section:Higher_order}

In the following sections we extend the notions of strong and
symmetric hyperbolicity to a certain type of higher order
in space systems. As Gundlach and Mart\'in-Garc\'ia 
in~\cite{GunGar05} we do not consider the most general first 
order in time, $N$th order in space system, but rather the subset 
for which a first order reduction exists. Here we describe 
these systems and establish our notation.

\subsection{FTNS systems}

\paragraph*{Notation:} We start by describing the notation 
that we use to present FT$N$S systems efficiently.
The equations of motion will be given for fields $v^\mu$,
where $v^\mu$ denotes a vector of fields which can appear
at most $N-\mu$ times differentiated in the FT$N$S system.
For reasons that will become clear later we also denote
fields with that property \emph{variables with $\mu$
implicit derivatives}.
To denote derivatives acting on $v^\nu$ we define for
$\mu=0,\dots,N-1$, $\nu=0,\dots,\mu$ and $\rho=1,\dots,\mu-\nu$ 
operators
\begin{align}
\hat A^\mu{}_\nu &:= (A^\mu{}_\nu)^{i_1\dots i_{\mu-\nu+1}}\p_{i_1\dots i_{\mu-\nu+1}},&\quad
\hat B^\mu{}_{\rho\,\nu} := ( B^\mu{}_{\rho\,\nu})^{i_1\dots i_{\mu-\nu-\rho+1}}\p_{i_1\dots i_{\mu-\nu-\rho+1}},
\nonumber
\end{align}
with constant matrices $(A^\mu{}_\nu)^{i_1\dots i_{\mu-\nu+1}}$ and
$(B^\mu{}_{\rho\,\nu})^{i_1\dots i_{\mu-\nu-\rho+1}}$.
Since the number of ``derivative indices'' (the indices denoted by lower
case Latin letters) in these matrices is
fixed through $\mu$, $\nu$ and $\rho$ we also use the abbreviations
\begin{align}
(A^\mu{}_\nu{})^{\ul{i}} &:=
(A^\mu{}_\nu{})^{i_1\dots i_{\mu-\nu+1}},&\quad
(A^\mu{}_\nu{})^{i_1\dots i_\sigma\ul{j}} :=
(A^\mu{}_\nu{})^{i_1\dots i_\sigma j_1\dots j_{\mu-\nu-\sigma+1}},\nonumber\\
(A^\mu{}_{\nu})^{i_1\dots i_\sigma\ul{i}} &:=
(A^\mu{}_{\nu})^{i_1\dots i_{\sigma} i_{\sigma+1}\dots i_{\mu-\nu+1}},\nonumber
\end{align}
i.e. an underlined lower case Latin letter means ``fill in
an appropriate number of derivative indices''. Analog notations
are used for the other objects that appear here.
The fields $v^\mu$ may also appear undifferentiated,
i.e. in the form $A^\mu{}_{\mu+1}v^{\mu+1}$ or
$B^\mu{}_{({\mu-\nu+1})\,\nu} v^\nu$. For efficiency
we use the same notation in that case:
\begin{align}
\hat A^{\mu}{}_{\mu+1}v^{\mu+1}
&:=A^{\mu}{}_{\mu+1}v^{\mu+1}=:(A^{\mu}{}_{\mu+1})^{\ul{i}}\p_{\ul{i}}v^{\mu+1},\nonumber\\
\hat B^{\mu}{}_{({\mu-\nu+1})\,\nu}v^\nu
&:=B^{\mu}{}_{({\mu-\nu+1})\,\nu}v^\nu
=:(B^{\mu}{}_{({\mu-\nu+1})\,\nu})^{\ul{i}}\p_{\ul{i}}v^\nu.
\end{align}

\paragraph*{Evolution equations:} We define an FT$N$S system 
as a system of equations of the form
\begin{align}
\p_tv^\mu &= \sum_{\nu=0}^{\mu+1}\hat A^\mu{}_\nu v^{\nu}
+\sum_{\nu=0}^{\mu}\sum_{\rho=1}^{\mu-\nu+1}\hat B^\mu{}_{\rho\,\nu} v^\nu+s^\mu\nonumber\\
\p_tv^{N-1} &= \sum_{\nu=0}^{N-1}\hat A^{N-1}{}_\nu v^{\nu}
+\sum_{\nu=0}^{{N-1}}\sum_{\rho=1}^{N-\nu}\hat B^{N-1}{}_{\rho\,\nu} v^\nu
+s^{N-1},\label{eq:system_FTNS}
\end{align}
with $\mu=0,\dots,N-2$ and source terms $s^\mu$, $s^{N-1}$
(the source terms do not contain the $v^\mu$). Note that FT$2$S
systems are the first order in time, second order in space systems
treated in~\cite{GunGar05} and FT$1$S systems are
fully first order systems. If we consider the equation of motion 
for~$v^\mu$ in~\eqref{eq:system_FTNS} then the left hand side,
$\p_t v^\mu$, is a first order derivative and in the right hand
side the highest derivative acting on $v^\nu$ has order $\mu-\nu+1$.
If we consider $v^\mu$ as a variable which contains $\mu$ derivatives
implicitly then the counting of derivatives gives at both sides $\mu+1$.
Therefore it is helpful to think of the $v^\mu$ in that way, which
explains our terminology.

\paragraph*{Principal part:} We will see that one can define 
strong and symmetric hyperbolicity of FT$N$S systems through 
the coefficients of the highest order derivatives 
in~\eqref{eq:system_FTNS}. Therefore we call
$\p_tv^\mu \simeq \sum_{\nu=0}^{\mu+1}\hat A^\mu{}_\nu v^{\nu},$ 
and~$\p_tv^{N-1}\simeq \sum_{\nu=0}^{N-1}\hat A^{N-1}{}_\nu v^{\nu}$,
with $\mu=0,\dots,N-2$ the \emph{principal part of the FT$N$S system}.
Furthermore we denote the matrix
\begin{align}
\nonumber
\mathcal A_N{}^p{}_{\ul{i}}{}^{\ul{j}} &=
\left(
(\Delta^N_{\mu\nu})^{p(\ul{j})}_{(\ul{i})\ul{k}} (\tilde
A^{\mu}{}_{\nu})^{\ul{k}}
\right)_{\mu=0,\dots,N-1}^{\nu=0,\dots,N-1},
\end{align}
with $(\ul{i})$ meaning symmetrization and
\begin{align}
\label{eq:FTNS_aux_defs}
(\tilde A^{\mu}{}_{\nu})^{\ul{i}} &:=
\left\{
\begin{array}{cc}
(A^{\mu}{}_{\nu})^{\ul{i}} & \mbox{ for }\nu\leq\mu+1\\
0 & \mbox{ for }\nu>\mu+1\;,
\end{array}
\right.\\
(\Delta^N_{\mu\nu})^{\ul{j}}_{\ul{i}\ul{k}} &:=
\left\{
\begin{array}{l}
\delta_{i_1}^{j_1}\dots\delta_{i_{N-\mu-1}}^{j_{N-\mu-1}}
\delta_{k_1}^{j_{N-\mu}}\dots\delta_{k_{\mu-\nu+1}}^{j_{N-\nu}}\\
\qquad \qquad \qquad \qquad \mbox{for } \mu\leq N-2, \nu\leq\mu\\
\delta_{i_1}^{j_1}\dots\delta_{i_{N-\mu-1}}^{j_{N-\mu-1}}
\quad \mbox{for } \mu\leq N-2, \nu=\mu+1\\
\delta_{k_1}^{j_{1}}\dots\delta_{k_{N-\nu}}^{j_{N-\nu}}
\quad \mbox{for } \mu = N-1, \nu\leq N-1\\
0 \quad \mbox{for } \nu>\mu+1
\end{array}
\right.\nonumber
\end{align}
for $\mu,\nu\leq N-1$ the \emph{principal matrix of the FT$N$S system}.
In the variables $u_{\ul{i}}=(\p_{i_1}\dots\p_{i_{N-\mu-1}}v^\mu)_{\mu=0,\dots,N-1}$
the principal part of the FT$N$S system can be written as
$\p_tu_{\ul{i}}\simeq \mathcal A_N{}^p{}_{\ul{i}}{}^{\ul{j}}
\p_pu_{\ul{j}}$.

\paragraph*{Principal symbol:} The \emph{principal symbol} of 
the FT$N$S system~\eqref{eq:system_FTNS} is $P_N^s = 
S^{N\,\ul{i}}\mathcal A_N{}^p{}_{\ul{i}}{}^{\ul{j}}s_p S^N_{\ul{j}}$,
where~$S^N_{\ul{j}}=
\mbox{diag}(s_{j_1}\dots s_{j_{N-1}},s_{j_1}\dots s_{j_{N-2}},
\dots,s_{j_1},1)$.

\section{Higher order strong hyperbolicity}
\label{section:FTNS_strong_hyp}

In this section we consider strong hyperbolicity of FTNS systems.
In analogy to the case of FT3S systems we introduce an iterative
reduction procedure, FT$N$S $\rightarrow$ FT$(N-1)$S $\rightarrow$
$\dots$ $\rightarrow$ FT1S and use this to define
strong hyperbolicity for FT$N$S systems without referring
to the reduction.

\subsection{Reduction to order $(N-1)$}
\label{sec:N-1_reduction}

\paragraph*{Reduction variables and auxiliary constraints:}
We begin with the description of reductions to order $(N-1)$.
The starting point is the FT$N$S system~\eqref{eq:system_FTNS}.

Using the same procedure that was described 
in detail for FT3S systems in section \ref{sec:2nd_order_reduction} 
we construct FT$(N-1)$S reductions of~\eqref{eq:system_FTNS}.
We define the reduction variables~$d_i := \p_i v^0$ 
and derive from~\eqref{eq:system_FTNS} their equation of motion:
\begin{align}
\label{eq:FTN-1S_red_var_EoM}
\p_t d_i &= (A^0{}_0)^j \p_i\p_j v^{0}+A^0{}_1 \p_i v^{1}
+\hat B_{1}^0{}_0 \p_i v^0+\p_i s^0.
\end{align}
The auxiliary constraints introduced with the new reduction
variable are
\begin{align}
c_i &:= \p_i v^0 - d_i,\quad\quad
c_{i_1\dots i_\sigma} :=
\p_{i_1}\dots\p_{i_{\sigma-1}} d_{i_\sigma} - \p_{(i_1}\dots\p_{i_{\sigma-1}} d_{i_\sigma)}.
\label{eq:FTN-1_aux_cnstr}
\end{align}
One can show that for $\sigma>2$ the constraints $c_{i_1\dots i_\sigma}$
can be written as linear combinations of derivatives of the
$c_{ij}$. The proof can be done through induction with the induction step
\begin{align}
\label{eq:FTN-1S_redundant_cnstr}
c_{i_1\dots i_\sigma} &=
\frac1\sigma \sum_{\mu=1}^{\sigma-1}
\p_{i_\mu}c_{i_1\dots i_{\rho-1}i_{\rho-2}\dots i_{\sigma}}
+\frac2{\sigma (\sigma-1)}\sum_{\nu=1}^{\sigma-1}
\p_{i_1}\dots\p_{i_{\nu-1}}\p_{i_{\nu+1}}\dots\p_{i_\sigma} c_{i_\mu i_\sigma}.
\end{align}

\paragraph*{FT$(N-1)$S reduction:} In analogy to 
section~\ref{sec:2nd_order_reduction} we come to FT$(N-1)$S 
reductions by adding the constraints~$c_i$ and~$c_{ij}$ as well as
their derivatives to
~\eqref{eq:system_FTNS} and~\eqref{eq:FTN-1S_red_var_EoM}.
If we restrict to those constraint 
additions which appear in the resulting FT$(N-1)$S principal part 
then we get the following class of FT$(N-1)$S systems
\begin{align}
\p_t v^0&=
(A^0{}_0)^k \p_k v^0 + (A^0{}_1) v^1 + (B^0{}_{10})v^0+s^0
+(D^0)^k c_k + (\bar D^0)^{kj} c_{kj},\nonumber\\
\p_t d_i&=
(A^0{}_0)^j\p_i d_j + A^0{}_1 \p_i v^1 + B^0{}_{10}\p_i v^0 + \p_i s^0
+(D)_i{}^k c_k + (\bar D)_i{}^{kj} c_{kj},\nonumber\\
\nonumber
\p_t v^\mu&=
(A^\mu{}_0)^{k_1\dots k_{\mu+1}} \p_{k_1}\dots \p_{k_\mu} d_{k_{\mu+1}}
+ \sum_{\nu=1}^{\mu+1}\hat A^\mu{}_\nu v^\nu
+ \sum_{\nu=0}^{\mu}\sum_{\rho=1}^{\mu-\nu+1}\hat B^\mu{}_{\rho\nu}v^\nu+s^\mu
\nonumber\\
&\quad+(D^\mu)^{k_1\dots k_{\mu}} \p_{k_1}\dots \p_{k_{\mu-1}} c_{k_\mu}
+ (\bar D^\mu)^{k_1\dots k_{\mu+1}} \p_{k_1}\dots \p_{k_{\mu-1}}c_{k_{\mu}k_{\mu+1}},
\nonumber\\
\nonumber
\p_t v^{N-1} &=
(A^{N-1}{}_0)^{k_1\dots k_{N}} \p_{k_1}\dots \p_{k_{N-1}} d_{k_{N}}+s^{N-1}
+ \sum_{\nu=1}^{N-1}\hat A^{N-1}{}_\nu v^\nu\\
\nonumber
&\quad+ \sum_{\nu=0}^{{N-1}}\sum_{\rho=1}^{N-\nu}\hat B^{N-1}{}_{\rho\nu}v^\nu
+(D^{N-1})^{k_1\dots k_{{N-1}}} \p_{k_1}\dots \p_{k_{N-2}} c_{k_{N-1}},\nonumber\\
&\quad+(\bar D^{N-1})^{k_1\dots k_{N}} \p_{k_1}\dots \p_{k_{N-2}}c_{k_{{N-1}}k_{N}}
\label{eq:N-1-reduction}
\end{align}
where $\mu=1,\dots,N-2$ and the matrices denoted~$D$ and~$\bar D$ 
are the reduction parameters. Due to the antisymmetry of $c_{ij}$ 
one can assume without loss of generality that the~$\bar D$ are 
antisymmetric in the last two indices. By applying this reduction 
procedure $(N-1)$ times we finally arrive at an FT1S system.\\ \\

\paragraph*{Auxiliary constraint evolution:}
By construction it is clear that there is a one-to-one correspondence 
between the solutions of~\eqref{eq:system_FTNS} and the solutions 
of~\eqref{eq:N-1-reduction} which satisfy the auxiliary
constraints~\eqref{eq:FTN-1_aux_cnstr}. The reason is that the 
constraint evolution system is closed:
\begin{align}
\p_t c_i &=
\left(
(A^0{}_0)^k + (D^0)^k
\right)\p_i c_k
+D_i{}^k c_k +(\bar D^0)^{kj} \p_i c_{kj}
+\bar D_i{}^{kj} c_{kj},\nonumber\\
\p_t c_{ij} &=
D_{[j}{}^k \p_{i]} c_k
+\bar D_{[j}{}^{kl} \p_{i]} c_{kl}.
\label{eq:FTN-1S_cnstr_evol}
\end{align}
Having~\eqref{eq:FTN-1S_redundant_cnstr} and~\eqref{eq:FTN-1S_cnstr_evol}
one can show by induction that $\p_t c_{i_1\dots i_\sigma}$ is equal to a
linear combination of the auxiliary constraints~\eqref{eq:FTN-1_aux_cnstr}
and their spatial derivatives.

\paragraph*{Principal part:}
The principal part of the FT$(N-1)$S 
system~\eqref{eq:N-1-reduction} is
\begin{align}
\p_t v^0 &\simeq
\left((A^0{}_0)^k + (D^0)^k\right) \p_k v^0 + (\bar D^0)^{kj} \p_k d_j,\nonumber\\
\p_t d_i &\simeq\left(B^0{}_{1\,0}\delta_i^k 
+ (D)_i{}^k\right)\p_k v^0 
+\left((A^0{}_0)^j\delta_i^k+(\bar D)_i{}^{kj}\right)\p_k d_j
+A^0{}_1 \p_i v^1,\nonumber\\
\nonumber
\p_t v^\mu &\simeq
\left((B^\mu{}_{1\,0})^{\ul{k}} + (D^\mu)^{\ul{k}}\right)
\p_{k_1}\dots \p_{k_\mu}v^0\nonumber\\
&\quad
+ \left(
(A^\mu{}_0)^{\ul{k}}+(\bar D^\mu{})^{\ul{k}}
\right)
\p_{k_1}\dots \p_{k_\mu} d_{k_{\mu+1}}
+\sum_{\nu=1}^{\mu+1}\hat A^\mu{}_\nu v^\nu,\nonumber\\
\nonumber
\p_t v^{N-1} &\simeq
\left(
(B^{N-1}{}_{1\,0})^{\ul{k}} + (D^{N-1})^{\ul{k}}
\right)
\p_{k_1}\dots \p_{k_{N-1}}v^0\nonumber\\
&\quad
+\left(
(A^{N-1}{}_0)^{\ul{k}}+ (\bar D^{N-1})^{\ul{k}}
\right) \p_{k_1}\dots \p_{k_{N-1}} d_{k_{N}}
+ \sum_{\nu=1}^{N-1}\hat A^{N-1}{}_\nu v^\nu,
\label{eq:N-1-reduction_PP}
\end{align}
where $\mu=1,\dots,N-2$.

For the ordering of variables $(v^0,d_i,v^1,\dots,v^{N-1})$
the principal matrix of the FT$(N-1)$S reduction
\eqref{eq:N-1-reduction} becomes
\begin{align}
&\mathcal A_{N-1}{}^p{}_{\ul{i}}{}^{\ul{j}}{}_i{}^j=\nonumber\\
\nonumber
&\left(
\begin{array}{ccc}
(\Delta^{N-1}_{00})^{p\ul{j}}_{\ul{i}k} ((A^0{}_0)^k + (D^0)^k)
&
(\Delta^{N-1}_{00})^{p\ul{j}}_{\ul{i}k} (\bar D^0)^{kj}
&
0\\
(\Delta^{N-1}_{00})^{p\ul{j}}_{\ul{i}k} ((B^0{}_{1\,0})\delta_i^k + (D)_i{}^k)
&
(\Delta^{N-1}_{00})^{p\ul{j}}_{\ul{i}k} \left((A^0{}_0)^{j}\delta_i^k
+(\bar D)_i{}^{kj}\right)
&
(\tilde\Delta^{N-1}_{0(\nu-1)})_{\ul{i}k}^{p\ul{j}} (\tilde A^0{}_\nu)\delta_i^k
\\
(\Delta^{N-1}_{(\mu-1)0})^{p\ul{j}}_{\ul{i}\ul{k}} ((B^\mu{}_{1\,0})^{\ul{k}} + (D^\mu)^{\ul{k}})
&
(\Delta^{N-1}_{(\mu-1)0})^{p\ul{j}}_{\ul{i}\ul{k}} \left((A^\mu{}_0)^{\ul{k}j}
+(\bar D^\mu)^{\ul{k}j}\right)
&
(\tilde\Delta^{N-1}_{(\mu-1)(\nu-1)})_{\ul{i}\ul{k}}^{p\ul{j}} (\tilde A^\mu{}_\nu)^{\ul{k}}
\end{array}
\right),
\end{align}
where $\mu,\nu=1,\dots,N-1$.

Note that
\begin{align}
(\Delta^{N-1}_{00})^{p\ul{j}}_{\ul{i}k} (A^0{}_0)^{j_{N-1}}\delta_{i_{N-1}}^k
&=(\Delta^{N}_{00})^{p\ul{j}}_{\ul{i}k} (A^0{}_0)^{k},\nonumber\\
(\Delta^{N-1}_{00})^{p\ul{j}}_{\ul{i}k} (A^0{}_1)\delta_{i_{N-1}}^k
&=(\Delta^{N}_{01})^{p\ul{j}}_{\ul{i}} (A^0{}_1),\nonumber\\
(\Delta^{N-1}_{(\mu-1)0})^{p\ul{j}}_{\ul{i}\ul{k}} (A^\mu{}_0)^{\ul{k}j_{N-1}}
&=(\Delta^{N}_{\mu 0})^{p\ul{j}}_{\ul{i}\ul{k}} (A^\mu{}_0)^{\ul{k}},\nonumber\\
(\tilde \Delta^{N-1}_{(\mu-1)(\nu-1)})^{\ul{k}}_{\ul{i}\ul{j}}
&=(\tilde \Delta^{N}_{\mu\nu})^{\ul{k}}_{\ul{i}\ul{j}}.\nonumber
\end{align}
Hence, if we rename $i\rightarrow i_{N-1}$ and $j\rightarrow j_{N-1}$ and
assume vanishing reduction parameters $\bar D$ then the FT$(N-1)$S principal
matrix has the FT$N$S principal matrix as a submatrix:
\begin{align}
\mathcal A_{N-1}{}^p{}_{\ul{i}}{}^{\ul{j}}{}_{i_{N-1}}{}^{j_{N-1}} &=
\left(
\begin{array}{cc}
*
&
0
\\
*
&
\mathcal A_{N}{}^p{}_{\ul{i}}{}^{\ul{j}}
\end{array}
\right).\nonumber
\end{align}

The FT$(N-1)$S principal symbol can be obtained by the appropriate contraction
of the principal matrix with a spatial vector~$s$:
\begin{align}
\label{eq:n-1_red_PSymbol}
P^s_{N-1}{}_i{}^j &=
\left(
\begin{array}{ccc}
((A^0{}_0)^k + (D^0)^k)s_k
&
(\bar D^0)^{kj}s_k
&
0\\
((B^0{}_{1\,0})\delta_i^k + (D)_i{}^k)s_k
&
\left((A^0{}_0)^{j}\delta_i^k
+(\bar D)_i{}^{kj}\right)s_k
&
(\tilde A^0{}_\nu)s_i
\\
((B^\mu{}_{1\,0})^{\ul{k}} + (D^\mu)^{\ul{k}})s^\mu_{\ul{k}}
&
\left((A^\mu{}_0)^{\ul{k}j}+(\bar D^\mu)^{\ul{k}j}\right)s^\mu_{\ul{k}}
&
(\tilde A^\mu{}_\nu)^{\ul{k}}s^{\mu-\nu+1}_{\ul{k}}
\end{array}
\right),
\end{align}
where $s^\nu_{\ul{k}}=s_{k_1}\dots s_{k_\nu}$.

\subsection{FT$N$S strong hyperbolicity}
\label{sec:FTNS_strong_hyp}

\paragraph*{Definitions of strong hyperbolicity:} Having 
defined reductions of FT$N$S systems to FT$(N-1)$S systems we 
now give two definitions of strong hyperbolicity for FT$N$S 
systems and show their equivalence. The first definition makes
use of the FT$(N-1)$S reduction.
\setcounter{subdefi}{1}
\renewcommand{\thedefi}{\arabic{defi}\alph{subdefi}}
\begin{defi}
\label{def:FTNS_strong_hyp_red}
The FT$N$S system~\eqref{eq:system_FTNS} is called \emph{FT$(N-1)$S strongly
hyperbolic} if there exists an FT$(N-1)$S reduction
\eqref{eq:N-1-reduction} which is FT$(N-1)$S strongly hyperbolic in the
sense of definition \ref{def:FTNS_strong_hyp_direct}.
\end{defi}
\addtocounter{defi}{-1}
\addtocounter{subdefi}{1}
The second definition does not rely on any reduction to lower order systems.
Note that for~$N=1$ it is consistent with the standard definition 
of strong hyperbolicity for fully first order systems~\cite{GusKreOli95}.
\begin{defi}
\label{def:FTNS_strong_hyp_direct}
The FT$N$S system~\eqref{eq:system_FTNS} is called FT$N$S strongly
hyperbolic if there exist a constant $M_N>0$ and a family of hermitian matrices
$H_N(s)$ such that
\begin{align}
H_N(s) P_N^s = (P_N^s)^\dag H_N(s), \quad
M_N^{-1}\,I  \leq H_N(s) \leq M_N\,I,\nonumber
\end{align}
where the matrix inequality is understood in the standard sense
\eqref{eq:matrix_inequality}.
\end{defi}
\renewcommand{\thedefi}{\arabic{defi}}

\paragraph*{Equivalence of the definitions:} We now demonstrate
that the two definitions of strong hyperbolicity are equivalent.
There is no major difference to the case of $N=3$ which was
discussed in section \ref{sec:FT3S_strong_hyp}.

\paragraph*{2+1 decomposition:} For the proof
we apply a 2+1 decomposition of the reduction variable~$d_i$. 
Let~$q_a^A$ be the orthogonal projector of~$s$, then the 
reduction variable is written as~$d_i = q_i^A d_A + s_i d_s$, 
where~$d_A s^A = 0$. With the state 
vector~$(v^0,d_A,d_s,v^1,\dots,v^{N-1})$ the principal 
symbol~\eqref{eq:n-1_red_PSymbol} becomes
\begin{align}
\label{eq:n-1_red_PSymbol_2+1}
&P^s_{N-1}{}_A{}^B =\\
\nonumber&
\left(
\begin{array}{cccc}
((A^0{}_0)^k + (D^0)^k)s_k
&
(\bar D^0)^{kj}s_kq_j^B
&
0
&
0\\
(D)_i{}^k s_k q^i_A
&
(\bar D)_i{}^{kj}s_k q_j^B q^i_A
&
0
&
0
\\
((B^0{}_{1\,0})\delta_i^k + (D)_i{}^k)s_k s^i
&
\left((A^0{}_0)^{j}\delta_i^k
+(\bar D)_i{}^{kj}\right)s_k q_j^B s^i
&
(A^0{}_0)^{j} s_j
&
\tilde A^0{}_\nu
\\
((B^\mu{}_{1\,0})^{\ul{k}} + (D^\mu)^{\ul{k}})s^\mu_{\ul{k}}
&
\left((A^\mu{}_0)^{\ul{k}j}+(\bar D^\mu)^{\ul{k}j}\right)s^\mu_{\ul{k}} q_j^B
&
(A^\mu{}_0)^{\ul{k}j}s^\mu_{\ul{k}} s_j
&
(\tilde A^\mu{}_\nu)^{\ul{k}}s^{\mu-\nu+1}_{\ul{k}}
\end{array}
\right).
\end{align}

\paragraph*{Definition \ref{def:FTNS_strong_hyp_red}
$\Rightarrow$ \ref{def:FTNS_strong_hyp_direct}: }
Assume that definition~\ref{def:FTNS_strong_hyp_red} is 
satisfied for an FT$(N-1)$S reduction~\eqref{eq:N-1-reduction}, i.e.
there exist a constant $M_{N-1}>0$ and a family of hermitian matrices
$H_{N-1}(s)^{AB}$ such that
\begin{align}
\nonumber
H_{N-1}(s)^{AB} P_{N-1}^s{}_{B}{}^{C} &= (P_{N-1}^s{}_{B}{}^{A})^\dag H_{N-1}(s)^{BC}, \\
\nonumber
M_{N-1}^{-1}\,I^{AB} & \leq H_{N-1}(s)^{AB} \leq M_{N-1}\,I^{AB},
\end{align}
where $I^{AB}$ is the appropriate identity.

Since~\eqref{eq:n-1_red_PSymbol_2+1}
is a block triangular matrix with the lower right diagonal block
\begin{align}
\left(
\begin{array}{cc}
(A^0{}_0)^{j} s_j
&
\tilde A^0{}_\nu
\\
(A^\mu{}_0)^{\ul{k}}s^\mu_{\ul{k}}
&
(\tilde A^\mu{}_\nu)^{\ul{k}}s^{\mu-\nu+1}_{\ul{k}}
\end{array}
\right) &= P_N^s,
\end{align}
the same arguments
used in section \ref{sec:FT3S_strong_hyp} can be applied
to show that in an appropriate decomposition of
$H_{N-1}(s)^{AB}$ the lower right block is a bounded symmetrizer
of~$P_N^s$. Hence, definition \ref{def:FTNS_strong_hyp_direct} is satisfied.

\paragraph*{Definition \ref{def:FTNS_strong_hyp_direct}
$\Rightarrow$ \ref{def:FTNS_strong_hyp_red}: }
Conversely, assuming definition \ref{def:FTNS_strong_hyp_direct}
is satisfied for an FT$N$S system~\eqref{eq:system_FTNS}
one can identify an FT$(N-1)$S reduction which is strongly hyperbolic.
We make the partial choice of reduction parameters
\begin{align}
\nonumber
(D^0)^k &= -(A^0{}_0)^k,&
(D)_i{}^k &= -(B^0{}_{1\,0})\delta_i^k,&
(D^\mu){}^{\ul{k}} &= -(B^\mu{}_{1\,0}){}^{\ul{k}},
\end{align}
for $\mu=1,\dots,N-1$. With this choice~\eqref{eq:n-1_red_PSymbol_2+1}
has the form
\begin{align}
P^s_{N-1}{}_A{}^B &=
\left(
\begin{array}{ccc}
0
&
0
&
0
\\
0
&
X_A^B
&
0
\\
0
&
Y^B
&
P_N^s
\end{array}
\right),\nonumber
\end{align}
where $X_A^B$ is the same matrix,~\eqref{eq:FT3S_XAB},
as in the FT3S case, $X_A^B = (\bar D)_i{}^{kj}s_k q_j^B q^i_A$,
and
\begin{align}
Y^B &=
\left(
\begin{array}{c}
\left((A^0{}_0)^{j}\delta_i^k
+(\bar D)_i{}^{kj}\right)s_k q_j^B s^i
\\
\left((A^\mu{}_0)^{\ul{k}j}+(\bar D^\mu)^{\ul{k}j}\right)s^\mu_{\ul{k}} q_j^B
\end{array}
\right).\nonumber
\end{align}

The same procedure that we used for FT3S systems in
section~\ref{sec:FT3S_strong_hyp}
allows the identification of a strongly hyperbolic FT$(N-1)$S reduction.
The key in this procedure is to choose~$
(\bar D)_i{}^{kj} = i\lambda \varepsilon_i{}^{jk}$,
where $\lambda\in\mathbb R$. With this the eigenvalues
of $X_A^B$ are $\pm\lambda$ and if $\lambda$
is sufficiently large then one can show that
definition \ref{def:FTNS_strong_hyp_red} is satisfied using the
assumption that the properties of the principal
symbol in definition \ref{def:FTNS_strong_hyp_direct} hold for $P_N^s$.
$\Box$

\subsection{Pseudo-differential reduction method}
\label{app:pseudo-red}

\paragraph*{Reduction variables:} To define strong hyperbolicity, 
in the literature a pseudo-differential reduction method is
used, see for example~\cite{NagOrtReu04}. With our calculations 
from section~\ref{sec:FTNS_strong_hyp} it is straightforward to apply
this method to FT$N$S systems as well. One takes a Fourier 
transformation in space of the FT$N$S 
system~\eqref{eq:system_FTNS} with wave 
number~$\omega_i = |\omega|s_i$. The Fourier transforms of 
the~$v^\mu$ are denoted $\hat v^\mu$ and we introduce a reduction 
variable~$\hat d^0 := i|\omega|\hat v^0$.

\paragraph*{Principal part:} Using the reduction variable 
the principal part of the Fourier transformed system (the terms 
with the highest order of $|\omega|$) can be written as
\begin{align}
\p_t \left(
\begin{array}{c}
(i|\omega|)^{N-2}\hat d^0\\
(i|\omega|)^{N-\mu-1}\hat v^\mu
\end{array}
\right)
 & \simeq i|\omega|
P_N^s
\left(
\begin{array}{c}
(i|\omega|)^{N-2}\hat d^0\\
(i|\omega|)^{N-\mu-1}\hat v^\mu
\end{array}
\right),\nonumber
\end{align}
where $P_N^s$ is the principal symbol of the FT$N$S 
system~\eqref{eq:system_FTNS} and the non principal terms not 
shown here are lower order in $|\omega|$. Applying this 
reduction~$(N-1)$ times results in a first order pseudo-differential 
system with principal symbol~$P_N^s$. Hence, using 
definition~\ref{def:FTNS_strong_hyp_direct}, an FT$N$S system is 
strongly hyperbolic if and only if there exists a strongly hyperbolic 
pseudo-differential reduction to order $(N-1)$.

\section{Higher order symmetric hyperbolicity}
\label{section:FTNS_sym_hyp}

In this section we show that one can extend the notion of
symmetric hyperbolicity to higher order in space systems.
For reasons discussed in section~\ref{sec:diff_reds} we
follow the strategy to employ a direct reduction
to first order.


\subsection{Reduction of FT$N$S systems to first order}
\label{section:Reduction}


\paragraph*{Reduction variables:}
We start with the description  of the reduction to first order
for the FT$N$S system~\eqref{eq:system_FTNS},
\begin{align}
\p_tv^\mu &= \sum_{\nu=0}^{\mu+1}\hat A^\mu{}_\nu v^{\nu}
+\sum_{\nu=0}^{\mu}\sum_{\rho=1}^{\mu-\nu+1}\hat B^\mu{}_{\rho\,\nu}
v^\nu+s^\mu,\nonumber\\
\p_tv^{N-1} &= \sum_{\nu=0}^{N-1}\hat A^{N-1}{}_\nu v^{\nu}
+\sum_{\nu=0}^{{N-1}}\sum_{\rho=1}^{N-\nu}\hat B^{N-1}{}_{\rho\,\nu} v^\nu
+s^{N-1},\label{eq:system_FTNS_app1}
\end{align}
with $\mu=0,\dots,N-2$. The reduction variables which we define 
are denoted $d^{\mu}_\nu$. The two indices have the following meaning:

\vspace{0.5em}
\begin{tabular}{|c|l|}
\hline
$\mu$ & the reduction variable refers to $v^\mu$ in the original
FT$N$S system\\
\hline
$\nu$ & the reduction variable has $\nu$ derivative indices ($1\leq\nu\leq
N-\mu-1$)\\
\hline
\end{tabular}
\vspace{0.5em}

The reduction variables are defined as
\begin{align}
\nonumber
(d^{\mu}_1)_i =
(d^{\mu}_1)_{\ul{i}} &:= \p_{i}v^{\mu},&
(d^{\mu}_\nu)_{i_1\dots i_\nu} = (d^{\mu}_\nu)_{\ul{i}}
&:=\p_{(i_1}(d^{\mu}_{\nu-1})_{i_2\dots i_\nu)},
\end{align}
where
$\mu=0,\dots,N-2$, $\nu=2,\dots,N-\mu-1$.
For convenience we also use the notation~$(d^{\mu}_0) = 
(d^{\mu}_0)_{\ul{i}} := v^{\mu}$. One finds that the important variables 
for the principal part of the first order reduction are those with the highest number of 
derivative indices, i.e. $d^{\mu}_{N-\mu-1}$. We abbreviate them as
$(d^\mu)_{\ul{i}} := (d^{\mu}_{N-\mu-1})_{\ul{i}}$,
where $\mu = 0,\ldots,N-2$.

\paragraph*{Unmodified equations of motion:} Using these 
definitions the equations of motion for the reduction variables 
can be derived by taking derivatives of~\eqref{eq:system_FTNS_app1}. 
One finds
\begin{align}
\label{eq:FTNS_rvars_eom}
\p_t(d^\mu_\nu)_{\ul{i}} &= \sum_{\rho=0}^{\mu+1}\hat A^\mu{}_\rho \, \p_{i_1}\dots\p_{i_\nu} v^{\rho}
+\sum_{\rho=0}^{\mu}\sum_{\sigma=1}^{\mu-\rho+1}\hat B^\mu{}_{\sigma\,\rho} \,
\p_{i_1}\dots\p_{i_\nu} v^\rho+\p_{i_1}\dots\p_{i_\nu} s^\mu,\nonumber
\end{align}
where~$\mu=0,\dots,N-2$, $\nu=1,\dots,N-\mu-1$ and we used
\begin{align}
\hat A^\mu{}_\rho &\, \p_{i_1}\dots\p_{i_\nu} v^{\rho}
:= (A^\mu{}_\rho)^{j_1\dots j_{\mu-\rho+1}} \p_{i_1}\dots\p_{i_\nu}\p_{j_1}\dots\p_{j_{\mu-\rho+1}} v^{\rho},\nonumber\\
\hat B^\mu{}_{\sigma\,\rho} &\, \p_{i_1}\dots\p_{i_\nu} v^{\rho}
:= (B^\mu{}_{\sigma\,\rho})^{j_1\dots j_{\mu-\rho-\sigma+1}} \p_{i_1}\dots\p_{i_\nu}\p_{j_1}\dots\p_{j_{\mu-\rho-\sigma+1}} v^{\rho}.
\end{align}
The terms~$\p_{i_1}\dots\p_{i_\nu} s^\mu$ in~\eqref{eq:FTNS_rvars_eom}
do not contain the~$d^\mu_\nu$ or~$v^\mu$ and can be seen as given
source terms.

\paragraph*{Auxiliary constraints:} The reduction variables 
are subject to the following first order auxiliary constraints
\begin{align}
(c^{\mu}_\nu){}_{i_1\dots i_\nu} &= (c^{\mu}_\nu)_{\ul{i}}:=
\p_{(i_1}(d^{\mu}_{\nu-1})_{i_2\dots i_\nu)}
-(d^{\mu}_\nu)_{i_1\dots i_\nu},\nonumber\\
(\bar c^{\mu}_\nu){}_{i_1\dots i_{\nu+1}} &= (\bar c^{\mu}_\nu){}_{\ul{i}}:=
\p_{i_1}(d^\mu_\nu)_{i_2\dots i_{\nu+1}} - \p_{(i_1}(d^\mu_\nu)_{i_2\dots i_{\nu+1})},
\label{eq:Reduction_Cnstr}
\end{align}
where $\mu=0,\dots,N-2$, $\nu=1,\dots,N-\mu-1$.

\paragraph*{First order reduction:} As before we ask now, which 
first order systems can be constructed by adding the 
constraints~\eqref{eq:Reduction_Cnstr} and their derivatives
to the right hand sides of~\eqref{eq:system_FTNS_app1} and
\eqref{eq:FTNS_rvars_eom}.

We note that lower order derivatives of the~$v^\mu$ 
(i.e. derivatives of order~$N-\mu-1$ or smaller) can be written
as linear combinations of the constraints, their derivatives
and undifferentiated reduction variables. For~$\mu=0,\dots,N-2$ 
and~$\nu=1,\dots,N-\mu-1$ one finds
\begin{align}
\label{eq:FTNS:lower_order_replacement}
\p_{i_1}\dots\p_{i_\nu}v^\mu &=
(d^\mu_\nu)_{\ul{i}}
+\sum_{\rho=0}^{\nu-1}\p_{i_1}\dots\p_{i_\rho}(c_{\nu-\rho}^\mu)_{i_{\rho+1}\dots i_{\nu}}
+\sum_{\rho=0}^{\nu-2}\p_{i_1}\dots\p_{i_\rho}(\bar c_{\nu-\rho-1}^\mu)_{i_{\rho+1}\dots i_{\nu}},
\end{align}
where the sums are understood to vanish when the upper bound is smaller than
the lower bound and the terms with $\rho=0$ should be interpreted as the
undifferentiated constraints.

One can prove this by induction over~$\nu$. 
For~$\nu=1$ we get~$\p_{i_1} v^\mu = (c^\mu_1)_{i_1} + (d^\mu_1)_{i_1},$
which is of the form~\eqref{eq:FTNS:lower_order_replacement}. Assuming
that~\eqref{eq:FTNS:lower_order_replacement} holds for a certain~$\nu$ 
we get
\begin{align}
\nonumber
\p_{i_1}&\dots\p_{i_{\nu+1}} v^\mu =
\p_{i_{1}}(d^\mu_{\nu})_{i_2\dots i_{\nu+1}}
+\p_{i_1}\sum_{\rho=1}^{\nu}\p_{i_2}\dots\p_{i_\rho}(c_{\nu-\rho+1}^\mu)_{i_{\rho+1}\dots i_{\nu+1}}\\
&\qquad
+\p_{i_1}\sum_{\rho=1}^{\nu-1}\p_{i_2}\dots\p_{i_\rho}(\bar c_{\nu-\rho}^\mu)_{i_{\rho+1}\dots i_{\nu+1}}.
\nonumber
\end{align}
In case $\nu<N-\mu-1$ the first term on the right hand side can be rewritten:
\begin{align}
\p_{i_{1}}(d^\mu_{\nu})_{i_2\dots i_{\nu+1}} &=
\p_{(i_{1}}(d^\mu_{\nu})_{i_2\dots i_{\nu+1})}
+(\bar c^\mu_{\nu})_{i_1\dots i_{\nu+1}}=(d^\mu_{\nu+1})_{\ul{i}}
+(c^\mu_{\nu+1})_{\ul{i}}
+(\bar c^\mu_{\nu})_{\ul{i}}.\nonumber
\end{align}
Hence, defining~$\tilde \nu = \nu + 1$ one gets
\begin{align}
\p_{i_1}&\dots\p_{i_{\tilde \nu}} v^\mu =
(d^\mu_{\tilde \nu})_{\ul{i}}
+\sum_{\rho=0}^{\tilde \nu-1}\p_{i_1}\dots\p_{i_\rho}(c_{\tilde \nu-\rho}^\mu)_{i_{\rho+1}\dots i_{\tilde \nu}}
+\sum_{\rho=0}^{\tilde \nu-2}\p_{i_1}\dots\p_{i_\rho}(\bar c_{\tilde \nu-\rho-1}^\mu)_{i_{\rho+1}\dots i_{\tilde \nu}}
\nonumber
\end{align}
for $\mu=0,\dots,N-2$ and $\tilde \nu = 1,\dots,N-\mu-1$.
Likewise one finds for $\mu=0,\dots,N-2$ and $\nu=N-\mu$
\begin{align}
\nonumber
\p_{i_1}&\dots\p_{i_{N-\mu}}v^\mu =
\p_{i_1}(d^\mu)_{i_2\dots i_{N-\mu}}
+\sum_{\rho=1}^{N-\mu-1}\p_{i_1}\dots\p_{i_\rho}(c_{N-\mu-\rho}^\mu)_{i_{\rho+1}\dots i_{N-\mu}}\\
&\quad
+\sum_{\rho=1}^{N-\mu-2}\p_{i_1}\dots\p_{i_\rho}(\bar c_{N-\mu-\rho-1}^\mu)_{i_{\rho+1}\dots i_{N-\mu}},
\nonumber
\end{align}
which is just the derivative of~\eqref{eq:FTNS:lower_order_replacement}
with $\nu=N-\mu-1$. This shows that when deriving a first order reduction
all lower order derivatives of the $v^\mu$ can be
completely absorbed into the constraint additions and that
up to constraint additions the highest order derivative of
$v^\mu$ becomes a first order symmetrized derivative of $d^\mu$.

\paragraph*{Reduction parameters:} The ambiguity of adding 
arbitrary linear combinations of the auxiliary 
constraints~\eqref{eq:Reduction_Cnstr} to the right hand sides 
of the first order system is parametrized by using 
\emph{reduction parameters}. We denote the constraint 
additions as
\begin{align}
\label{eq:1st_cnstr_additions}
D^{X\,\sigma}{}_{\nu} c^{\nu}_\sigma
&:= (D^{X\,\sigma}{}_{\nu})^{i_1\dots i_\sigma}
(c^{\nu}_\sigma)_{i_1\dots i_\sigma},&
\bar D^{X\,\sigma}{}_{\nu} \bar c^{\nu}_\sigma
&:= (\bar D^{X\,\sigma}{}_{\nu})^{i_1\dots i_{\sigma+1}}
(\bar c^{\nu}_\sigma)_{i_1\dots i_{\sigma+1}},
\end{align}
where
$\nu=0,\dots,N-2$ and $\sigma=1,\dots,N-\nu-1$. Depending on the equation
where we add those constraints the index $X$ is either a single index
(in the case of constraint additions to the right hand sides of $v^\mu$)
or an index-tuple $(\mu,\lambda,i_1,\dots,i_\lambda)$ (in the right hand
sides of~$d^{\mu}_{\lambda}$). The matrices~$(D^{X\,\sigma}{}_{\nu})^{i_1\dots i_\sigma}$
and~$(\bar D^{X\,\sigma}{}_{\nu})^{i_1\dots i_{\sigma+1}}$
are the reduction parameters. Without loss of generality we
assume the symmetry properties
\begin{align}
\label{eq:FTNS:red_var_symm}
\nonumber
(D^{X\,\sigma}{}_{\nu})^{i_1\dots i_{\sigma}} &= (D^{X\,\sigma}{}_{\nu})^{(i_1\dots i_{\sigma})},&
(\bar D^{X\,\sigma}{}_{\nu})^{(i_1\dots i_{\sigma+1})} &=0,\\
(\bar D^{X\,\sigma}{}_{\nu})^{i_1i_2\dots i_{\sigma+1}} &=
(\bar D^{X\,\sigma}{}_{\nu})^{i_1(i_2\dots i_{\sigma+1})}.&
\end{align}
The constraint additions on the different equations
are independent of each other. We use the short notation
\begin{align}
\label{eq:cnstr_additions}
C^X &=
\sum_{\nu=0}^{N-2}\sum_{\sigma=1}^{N-\nu-1} D^{X\,\sigma}{}_{\nu} c^{\nu}_\sigma
+\sum_{\nu=0}^{N-2}\sum_{\sigma=1}^{N-\nu-1} \bar D^{X\,\sigma}{}_{\nu} \bar c^{\nu}_\sigma,
\end{align}
where $X$ has the same meaning as in~\eqref{eq:1st_cnstr_additions}.

\paragraph*{Reduced equations of motion:} With these findings 
the right hand sides for the~$v^\mu$ in the first order 
reductions of~\eqref{eq:system_FTNS_app1} have the form
\begin{align}
\nonumber
\p_t v^\mu &=
C^\mu+s^\mu
+\sum_{\nu=0}^{\mu+1} (A^\mu{}_\nu)^{\ul{j}}(d^\nu_{\mu-\nu+1})_{\ul{j}}
+\sum_{\nu=0}^{\mu}\sum_{\rho=1}^{\mu-\nu+1}
(B^\mu{}_{\rho\,\nu})^{\ul{j}}(d^\nu_{\mu-\nu-\rho+1})_{\ul{j}},\\
\nonumber
\p_t v^{N-2} &=
C^{N-2}+s^{N-2}
+\sum_{\nu=0}^{N-2} (A^{N-2}{}_\nu)^{\ul{j}}(d^\nu)_{\ul{j}}
+ (A^{N-2}{}_{N-1})v^{N-1}\\
\nonumber &\qquad
+\sum_{\nu=0}^{N-2}\sum_{\rho=1}^{N-\nu-1}
(B^\mu{}_{\rho\,\nu})^{\ul{j}}(d^\nu_{\mu-\nu-\rho+1})_{\ul{j}},\\
\p_t v^{N-1} &= \sum_{\nu=0}^{{N-2}}(A^{N-1}{}_\nu)^{j\ul{i}}
\p_{j} (d^{\nu})_{\ul{i}}+(A^{N-1}{}_{N-1})^j \p_j v^{N-1}
+C^{N-1}+s^{N-1},
\label{eq:reduced_system}
\end{align}
for $\mu=0,\dots,N-3$. Likewise one finds the equations of motion 
for the reduction variables in the first order reduction
\begin{align}
\p_t(d^{\mu}_\sigma)_{\ul{i}} &= (C^{\mu}_\sigma)_{\ul{i}} +
\p_{i_1\dots i_\sigma}^\sigma s^{\mu}
+\sum_{\nu=0}^{\mu+1} (A^\mu{}_\nu)^{\ul{j}}
(d^\nu_{\mu+\sigma-\nu+1})_{i_1\dots i_\sigma \ul{j}}\nonumber\\
&\quad
+\sum_{\nu=0}^{\mu}\sum_{\rho=1}^{\mu-\nu+1}
(B^\mu{}_{\rho\,\nu})^{\ul{j}}(d^\nu_{\mu+\sigma-\nu-\rho+1})_{i_1\dots i_\sigma \ul{j}},
\nonumber\\
\p_t (d^\mu)_{\ul{i}} &=
\sum_{\nu=0}^{\mu+1}
(A^\mu{}_\nu)^{\ul{j}}(\Delta^N_{\mu\nu})^{p\ul{k}}_{\ul{i}\ul{j}}\p_{(p} (d^\nu)_{\ul{k})}
+(C^{\mu}_{N-\mu-1})_{\ul{i}}+\p_{i_1\dots i_{N-\mu-1}}^{N-\mu-1} s^{\mu}\nonumber\\
&\quad
+\sum_{\nu=0}^{\mu}\sum_{\rho=1}^{\mu-\nu+1}
(B^\mu{}_{\rho\,\nu})^{\ul{j}}(d^\nu_{N-\nu-\rho})_{i_1\dots i_{N-\mu-1} \ul{j}},
\nonumber\\
\nonumber
\p_t (d^{N-2})_{i} &=
\sum_{\nu=0}^{{N-2}}(A^{N-2}{}_\nu)^{\ul{j}}\p_{(i}(d^\nu)_{\ul{j})}
+A^{N-2}{}_{{N-1}}\p_{i}v^{N-1}+(C^{N-2}_{1})_{{i}}+\p_i s^{N-2}\\
&\quad+\sum_{\nu=0}^{N-2}\sum_{\rho=1}^{N-\nu-1}
(B^{N-2}{}_{\rho\,\nu})^{\ul{j}}(d^\nu_{N-\nu-\rho})_{i\ul{j}},
\label{eq:1st_order_red}
\end{align}
where~$\mu=0,\dots,N-3$ and~$\sigma=1,\dots,N-\mu-2$.
The~$C^{\mu}_{\sigma}$ can be read off from~\eqref{eq:cnstr_additions},
and in~\eqref{eq:1st_order_red} we used the symbol
$(\Delta^N_{\mu\nu})^{\ul{k}}_{\ul{i}\ul{j}}$
which is defined in~\eqref{eq:FTNS_aux_defs}. We call a system of 
the form~\eqref{eq:reduced_system},\eqref{eq:1st_order_red}
a~\emph{first order reduction} or~\emph{FT1S reduction} of 
the FT$N$S system~\eqref{eq:system_FTNS_app1}.

\paragraph*{Principal part:} We now write the principal part 
of the first order 
reduction~\eqref{eq:reduced_system},\eqref{eq:1st_order_red} 
in a standard form. The terms that contain derivatives in the 
constraint additions are
\begin{align}
\nonumber
C^X &\simeq
\sum_{\nu=0}^{N-2}\sum_{\sigma=0}^{N-\nu-2}
(D^{X\,(\sigma+1)}{}_{\nu})^{i_1\dots i_{\sigma+1}}
\p_{i_1}(d^{\nu}_{\sigma})_{i_2\dots i_{\sigma+1}}\\
\nonumber&\quad
+\sum_{\nu=0}^{N-2}\sum_{\sigma=1}^{N-\nu-1}
(\bar D^{X\,\sigma}{}_{\nu})^{i_1\dots i_{\sigma+1}}
\p_{i_1}(d^{\nu}_\sigma)_{i_2\dots i_{\sigma+1}}\\
&=
\sum_{\nu=0}^{N-2}(D^{X\,1}{}_{\nu})^{i_1}
\p_{i_1}v^{\nu}
+\sum_{\nu=0}^{N-3}\sum_{\sigma=1}^{N-\nu-2}
(\tilde D^{X\,\sigma}{}_{\nu})^{i_1\dots i_{\sigma+1}}
\p_{i_1}(d^{\nu}_{\sigma})_{i_2\dots i_{\sigma+1}}\nonumber\\
&\quad
+\sum_{\nu=0}^{N-2}
(\bar D^{X\,(N-\nu-1)}{}_{\nu})^{i_1\dots i_{N-\nu}}
\p_{i_1}(d^{\nu})_{i_2\dots i_{N-\nu}},\nonumber
\end{align}
where
$(\tilde D^{X\,\sigma}{}_{\nu})^{i_1\dots i_{\sigma+1}}
=(\tilde D^{X\,\sigma}{}_{\nu})^{\ul{i}}
:= (D^{X\,(\sigma+1)}{}_{\nu})^{i_1\dots i_{\sigma+1}}
+(\bar D^{X\,\sigma}{}_{\nu})^{i_1\dots i_{\sigma+1}}$,
and we used the symmetry properties~\eqref{eq:FTNS:red_var_symm}
of the reduction parameters. The symbol $\simeq$ means
equality up to terms without derivatives and
$X$ has the same meaning as in~\eqref{eq:1st_cnstr_additions}.
We write the state vector as~$
u_{\ul{i}}:=
\left(
(d^{\tilde \mu}_\sigma)_{\ul{i}},
 v^{\mu},
 (d^{\mu})_{\ul{i}},
 w
\right)^\dag$,
where the bounds for the indices are~$\mu = 0,\dots,N-2$, $\tilde 
\mu=0,\dots,N-3$ and~$\sigma=1,\dots,N-\tilde \mu-2$. The principal 
part of the system~\eqref{eq:reduced_system},\eqref{eq:1st_order_red}
is then~$\p_t u_{\ul{i}} \simeq\mathcal A_1{}^p{}_{\ul{i}}{}^{\ul{j}}\p_pu_{\ul{j}}$,
where
\begin{align}
\label{eq:PP_matrix_1st}
&\mathcal A_1{}^p{}_{\ul{i}}{}^{\ul{j}} =\\
\nonumber
&\left(
\begin{array}{cccc}
(\tilde D{}^{\tilde \mu}{}_\sigma{}^\rho{}_{\tilde\nu}){}_{\ul{i}}{}^{p\ul{j}}
&
(D^{\tilde \mu}{}_\sigma{}^1{}_{\nu})_{\ul{i}}{}^p
&
(\bar D{}^{\tilde \mu}{}_{\sigma}{}^{N-\nu-1}{}_{\nu}){}_{\ul{i}}{}^{p\ul{j}}
&
0
\\
(\tilde D{}^{\mu\,\rho}{}_{\tilde \nu})^{p\ul{j}}
&
(D^{\mu\,1}{}_{\nu})^p
&
(\bar D{}^{\mu\,(N-\nu-1)}{}_{\nu}){}^{p\ul{j}}
&
0
\\
(\tilde D{}^{\mu}{}_{N-\mu-1}{}^{\rho}{}_{\tilde \nu}){}_{\ul{i}}{}^{p\ul{j}}
&
(D^{\mu}{}_{N-\mu-1}{}^1{}_{\nu})_{\ul{i}}{}^p
&
(\tilde A^{\mu}{}_{\nu})^{\ul{k}}(\tilde \Delta^N_{\mu\nu})^{(p\ul{j})}_{\ul{i}\ul{k}}
+ (\bar D{}^{\mu}{}_{N-\mu-1}{}^{N-\nu-1}{}_{\nu}){}_{\ul{i}}{}^{p\ul{j}}
&
0
\\
(\tilde D{}^{(N-1)\,\rho}{}_{\tilde \nu}){}^{p\ul{j}}
&
(D^{(N-1)\,1}{}_{\nu}){}^p
&
(\tilde A^{N-1}{}_{\nu})^{p\ul{j}}
+ (\bar D{}^{(N-1)\,(N-\nu-1)}{}_{\nu}){}^{p\ul{j}}
&
(A^{N-1}{}_{N-1})^{p}
\end{array}
\right)
\end{align}
and we used definition~\eqref{eq:FTNS_aux_defs}
for the symbols $\tilde A_{\mu\nu}$ and  $\tilde \Delta^N_{\mu\nu}$.
The range of the various indices in this expression is~$\mu,\nu=0,
\dots,N-2$, $\tilde \mu,\tilde \nu=0,\dots,N-3$, $\sigma=1,
\dots,N-\tilde \mu-2$ and~$\rho=1,\dots,N-\tilde\nu-2$.

\paragraph*{Auxiliary constraint evolution:}
Having defined what we mean by first order reductions of the FT$N$S
system~\eqref{eq:system_FTNS_app1} we note that again there is a
one-to-one correspondence between the solutions of the first order
reduction~\eqref{eq:reduced_system},\eqref{eq:1st_order_red}
which satisfy the auxiliary constraints
\eqref{eq:Reduction_Cnstr} and the solutions of the original FT$N$S
system~\eqref{eq:system_FTNS_app1}. This property of the reduced 
systems is a consequence of the construction procedure, which 
leads to a closed constraint evolution system. To see that the 
constraint evolution system is closed is straightforward. One 
just uses equation~\eqref{eq:FTNS:lower_order_replacement} to 
express the reduction variables by derivatives of the~$v^\mu$ 
and constraints. In the right hand sides of the constraint
evolution system the derivatives of the~$v^\mu$ cancel due to 
their symmetry in the derivative indices. This leads to the 
closed constraint evolution system. However, one obtains 
very lengthy expressions, so we suppress the details.

\subsection{FT$N$S symmetric hyperbolicity}

\paragraph*{Definitions of symmetric hyperbolicity:} To get 
definitions of symmetric hyperbolicity for FT$N$S systems
we generalize the second order definitions given in
\cite{GunGar05}. We start by defining candidate symmetrizers.
\setcounter{subdefi}{1}
\renewcommand{\thedefi}{\arabic{defi}\alph{subdefi}}
\begin{defi}
\label{def:cand_sym_FTNS}
Given an FT$N$S system~\eqref{eq:system_FTNS_app1} we call a Hermitian
matrix $H_N^{\ul{i}\,\ul{j}} = H_N^{(\ul{i})\,(\ul{j})}$ such that the product matrix
$S^N_{\ul i}H_N^{\ul{i}\,\ul{k}}{\mathcal A}_N^p{}_{\ul{k}}{}^{\ul{j}}s_p S^N_{\ul j}$,
is Hermitian for every $s$ an \emph{FT$N$S candidate symmetrizer}.
\end{defi}
\addtocounter{defi}{-1}
\addtocounter{subdefi}{1}
When we refer to lower order systems then we require the
existence of a first order reduction such that
there is a candidate symmetrizer in the usual first order sense:
\begin{defi}
\label{def:1st_cand_sym_FTNS}
We call a Hermitian matrix $H_1^{\ul{i}\,\ul{j}} = H_1^{(\ul{i})\,(\ul{j})}$ a
\emph{first order candidate symmetrizer} of~\eqref{eq:system_FTNS_app1} if there
exists a first order reduction
\eqref{eq:reduced_system},\eqref{eq:1st_order_red}
such that the product
$H_1^{\ul{i}\,\ul{k}}{\mathcal A}_1^p{}_{\ul{k}}{}^{\ul{j}}s_p,$
is Hermitian for every $s$.
\end{defi}
In both cases we call a positive definite candidate symmetrizer
a \emph{symmetrizer}. With this it is straightforward to define 
symmetric hyperbolicity with and without reference to a first 
order reduction
\setcounter{subdefi}{1}
\begin{defi}
\label{def:sym_hyp_FTNS}
The FT$N$S system~\eqref{eq:system_FTNS_app1} is called 
\emph{FT$N$S symmetric hyperbolic} if there exists a positive 
definite FT$N$S candidate symmetrizer.
\end{defi}
\addtocounter{defi}{-1}
\addtocounter{subdefi}{1}

\begin{defi}
\label{def:1st_sym_hyp_FTNS}
The FT$N$S system~\eqref{eq:system_FTNS_app1} is called
\emph{first order symmetric hyperbolic} if there exists a positive definite
first order candidate symmetrizer.
\end{defi}
\renewcommand{\thedefi}{\arabic{defi}}

\paragraph*{Relationship between the definitions:}
Now we show for arbitrary~$N$ that definition
\ref{def:1st_sym_hyp_FTNS} implies \ref{def:sym_hyp_FTNS}.
The proof of the reverse direction for arbitrary $N$
involves very complicated expressions. We
show in {\tt automatic\_construction\_of\_J.nb}
\footnote{{\tt http://www.tpi.uni-jena.de/\~{}hild/FTNS.tgz}}
that for $N=3$ it is indeed true that \ref{def:sym_hyp_FTNS}
implies \ref{def:1st_sym_hyp_FTNS}. For $N\leq \maxN$
we checked this using the same computer algebra method. 
However, whether the statement holds for arbitrary~$N$ is an open 
question.

\paragraph*{Construction of $N$th order from first order candidates: }
Let $H_1^{\ul{i}\,\ul{j}}$ be the candidate symmetrizer
of a first order reduction with principal matrix
${\mathcal A}_1^p{}_{\ul{k}}{}^{\ul{j}}$. We group the state vector
as
$u_{\ul{i}}:=
\left(
(d^{\tilde \mu}_\sigma)_{\ul{i}},
 v^{\mu}\,\,|\,\,
 (d^{\mu})_{\ul{i}},
 w
\right)^\dag
$
and in this way decompose $H_1$ and $\mathcal A_1$
consistently into
\begin{align}
\label{eq:FTNS_2x2_decomp_1st}
H_1^{\ul{i}\,\ul{j}} &=
\left(
\begin{array}{cc}
H_{11}^{\ul{i}\,\ul{j}} &
H_{12}^{\ul{i}\,\ul{j}} \\
H_{21}^{\ul{i}\,\ul{j}} &
H_{22}^{\ul{i}\,\ul{j}}
\end{array}
\right),&
{\mathcal A}_1^p{}_{\ul{k}}{}^{\ul{j}}
&=
\left(
\begin{array}{cc}
{\mathcal A}_{11}^p{}_{\ul{k}}{}^{\ul{j}} &
{\mathcal A}_{12}^p{}_{\ul{k}}{}^{\ul{j}} \\
{\mathcal A}_{21}^p{}_{\ul{k}}{}^{\ul{j}} &
{\mathcal A}_{22}^p{}_{\ul{k}}{}^{\ul{j}}
\end{array}
\right),
\end{align}
where
\begin{align}
{\mathcal A}_{12}^p{}_{\ul{k}}{}^{\ul{j}} &=
\left(
\begin{array}{cc}
(\bar D{}^{\tilde \mu}{}_{\sigma}{}^{N-\nu-1}{}_{\nu}){}_{\ul{i}}{}^{p\ul{j}}
&
0
\\
(\bar D{}^{\mu\,(N-\nu-1)}{}_{\nu}){}^{p\ul{j}}
&
0
\end{array}
\right),\nonumber\\
{\mathcal A}_{22}^p{}_{\ul{k}}{}^{\ul{j}} &=
\left(
\begin{array}{cc}
(\tilde A^{\mu}{}_{\nu})^{\ul{k}}(\tilde \Delta^N_{\mu\nu})^{p(\ul{j})}_{(\ul{i})\ul{k}}
+ (\bar D{}^{\mu}{}_{N-\mu-1}{}^{N-\nu-1}{}_{\nu}){}_{\ul{i}}{}^{p\ul{j}}
&
0
\\
(\tilde A^{N-1}{}_{\nu})^{p\ul{j}}
+ (\bar D{}^{(N-1)\,(N-\nu-1)}{}_{\nu}){}^{p\ul{j}}
&
(A^{N-1}{}_{N-1})^{p}
\end{array}
\right),\nonumber
\end{align}
i.e. such that~$\mathcal A_{22}$ is the lower right $2\times 2$
block of~\eqref{eq:PP_matrix_1st}. In this decomposition
the lower right block of the product
$H_1^{\ul{i}\,\ul{k}}{\mathcal A}_1^p{}_{\ul{k}}{}^{\ul{j}}$ is
\begin{align}
\label{eq:FTNS_lr_hermit_block}
H_{21}^{\ul{i}\,\ul{k}}{\mathcal A}_{12}^p{}_{\ul{k}}{}^{\ul{j}}+
H_{22}^{\ul{i}\,\ul{k}}{\mathcal A}_{22}^p{}_{\ul{k}}{}^{\ul{j}}.
\end{align}
Hence, the matrix~\eqref{eq:FTNS_lr_hermit_block} is Hermitian
for every $p$, because it is a principal minor of
$H_1^{\ul{i}\,\ul{k}}{\mathcal A}_1^p{}_{\ul{k}}{}^{\ul{j}}$.

Moreover, because $S^N_{\ul{i}}$ is Hermitian for every $s$,
we get that~\eqref{eq:FTNS_lr_hermit_block} contracted from
left and right with $S^N_{\ul{i}}$ is Hermitian for every $p$
as well. Thus,
$S^N_{\ul{i}} H_{21}^{\ul{i}\,\ul{k}}{\mathcal A}_{12}^p{}_{\ul{k}}{}^{\ul{j}} s_p S^N_{\ul{j}}+
S^N_{\ul{i}} H_{22}^{\ul{i}\,\ul{k}}{\mathcal A}_{22}^p{}_{\ul{k}}{}^{\ul{j}} s_p S^N_{\ul{j}}$,
is Hermitian for every $s$. On the other hand
${\mathcal A}_{12}^p{}_{\ul{k}}{}^{\ul{j}} s_p S^N_{\ul{j}}= 0$ and 
${\mathcal A}_{22}^p{}_{\ul{k}}{}^{\ul{j}} s_p S^N_{\ul{j}}
={\mathcal A}_{N}^p{}_{\ul{k}}{}^{\ul{j}} s_p S^N_{\ul{j}}$,
because the symmetric part of the reduction parameters contained
in ${\mathcal A}_{12}$ and ${\mathcal A}_{22}$ vanishes.

Since $H_{22}$ is on the diagonal of $H_1$ it is Hermitian as well.
Thus, with the identification
$H_{N}^{\ul{i}\,\ul{k}} = H_{22}^{\ul{i}\,\ul{k}}$
there exists an~$N$th order candidate symmetrizer.

\paragraph*{Positivity of the FT$N$S candidate symmetrizer: }
Moreover, if $H_1^{\ul{i}\,\ul{j}}$ is positive definite
then also $H_{22}^{\ul{i}\,\ul{k}}$ is positive definite, because it
is a principal minor. Hence, if there exists a first order 
reduction of~\eqref{eq:system_FTNS_app1} which is symmetric 
hyperbolic then~\eqref{eq:system_FTNS_app1} is also FT$N$S 
symmetric hyperbolic with the symmetrizer~$H_{N}^{\ul{i}\,\ul{k}} 
= H_{22}^{\ul{i}\,\ul{k}}$.

\paragraph*{Construction of a symmetric hyperbolic first order 
reduction:} Now, for the reverse direction we assume a given
FT$N$S symmetrizer, $H_{N}^{\ul{i}\,\ul{j}}$, and would like to show that
there exists a first order reduction with symmetrizer
\begin{align}
H_1^{\ul{i}\,\ul{j}} &=
\left(
\begin{array}{c|c}
\begin{array}{ccc}
\Gamma_{\rho}^{(i_1\dots i_{\rho})\,(j_1\dots j_{\rho})} & \cdots & 0\\
\vdots & \ddots & \vdots\\
0 & \cdots & 1\\
\end{array}
& 0\\
\hline
0 & H_N^{\ul{i}\,\ul{j}}
\end{array}
\right).\nonumber
\end{align}
(in the $2\times 2$ decomposition~\eqref{eq:FTNS_2x2_decomp_1st})
with
$\Gamma_{\rho}^{i_1\dots i_{\rho}\,j_1\dots j_{\rho}}
=\gamma^{i_1j_1}\dots\gamma^{i_\rho j_\rho}
$
and $\rho$ such that the $\Gamma_\rho^{\ul{i}\,\ul{j}}$ has the appropriate
number of indices. Obviously positivity of
$H_{N}^{\ul{i}\,\ul{j}}$ implies positivity of 
$H_1^{\ul{i}\,\ul{j}}$, i.e. we only need to show the conservation 
property.

To identify an appropriate reduction to first order we first make
the partial choice of reduction parameters
$
(D^{X\,\sigma}{}_{\nu})^{\ul{i}}= 
(\bar D^{X\,\sigma}{}_{\nu})^{\ul{i}} = 0,
$
for $\nu = 0,\dots,N-3$ and $\sigma = 1,\dots,N-\nu-2$, i.e. only
the reduction parameters which correspond to the constraint additions
with the highest number of derivative indices remain. As in
\eqref{eq:1st_cnstr_additions}
$X$ denotes either a single index $\mu=0,\dots,N-1$
or an index tuple $(\mu,\lambda,i_1,\dots,i_\lambda)$ with
$\mu=0,\dots,N-2$ and $\lambda = 1,\dots,N-\mu-1$.

With that choice most of the components of $\mathcal A_1^p{}_{\ul{i}}{}^{\ul{j}}$
vanish and the statement which needs to be shown is that
there exist reduction parameters such that
$H_N^{\ul{i}\,\ul{j}}\tilde{\mathcal A}_N^p{}_{\ul{j}}{}^{\ul{k}}s_p$
is Hermitian for every $s$, where
$\tilde{\mathcal A}_N^p{}_{\ul{j}}{}^{\ul{k}}=
{\mathcal A}_N^p{}_{\ul{j}}{}^{\ul{k}}
+{\bar D}_N^p{}_{\ul{j}}{}^{\ul{k}}$,
and
\begin{align}
{\bar D}_N^p{}_{\ul{j}}{}^{\ul{k}} &=
\left(
\begin{array}{cc}
(\bar D{}^{\mu}{}_{N-\mu-1}{}^{N-\nu-1}{}_{\nu}){}_{\ul{j}}{}^{p\ul{k}}
&
0
\\
(\bar D{}^{(N-1)\,(N-\nu-1)}{}_{\nu}){}^{p\ul{k}}
&
0
\end{array}
\right).\nonumber
\end{align}

We define
\begin{align}
\left(T_{\mu\nu}^{p\,\ul{i}\,\ul{k}}\right)_{\mu=0,\dots,N-1}^{\nu=0,\dots,N-1} =
T_N^{p\,\ul{i}\,\ul{k}} := H_N^{\ul{i}\,\ul{j}}{\mathcal A}_N^p{}_{\ul{j}}{}^{\ul{k}},\nonumber\\
\left(J_{\mu\nu}^{p\,\ul{i}\,\ul{k}}\right)_{\mu=0,\dots,N-1}^{\nu=0,\dots,N-1} =
J_N^{p\,\ul{i}\,\ul{k}} := H_N^{\ul{i}\,\ul{j}}{\bar D}_N^p{}_{\ul{j}}{}^{\ul{k}},
\label{eq:FTNS:def_TJ}
\end{align}
where it is understood that decomposition of $T_N$ and $J_N$ into $T_{\mu\nu}$
and $J_{\mu\nu}$ is the one induced by the original FT$N$S system
\eqref{eq:system_FTNS_app1}.

One finds that the hermiticity of $H_N^{\ul{i}\,\ul{j}}\tilde{\mathcal A}_N^p{}_{\ul{j}}{}^{\ul{k}}s_p$
is equivalent to
\begin{align}
\label{eq:FTNS:conserv_TJ}
T_N^{p\,\ul{i}\,\ul{k}} + J_N^{p\,\ul{i}\,\ul{k}}
&= \left(T_N^{p\,\ul{i}\,\ul{k}} + J_N^{p\,\ul{i}\,\ul{k}}\right)^\dag\quad \forall p.
\end{align}
In components equation~\eqref{eq:FTNS:conserv_TJ} is
\begin{align}
\label{eq:FTNS:conserv_TJ_comp}
J_{\mu\nu}^{p\,\ul{i}\,\ul{j}} + T_{\mu\nu}^{p\,\ul{i}\,\ul{j}}
&=
J_{\nu\mu}^{\dag p\,\ul{j}\,\ul{i}} + T_{\nu\mu}^{\dag p\,\ul{j}\,\ul{i}}\quad\forall p.
\end{align}

From definition~\eqref{eq:FTNS:def_TJ} we see that the $J_{\mu\nu}$
need to satisfy certain symmetry conditions:
\begin{align}
\label{eq:FTNS:Jmunu_symmetry}
J_{\mu\nu}^{(p|\,\ul{i}\,|\ul{j})} &= 0, &
J_{\mu\nu}^{p\,\ul{i}\,\ul{j}} &= J_{\mu\nu}^{p\,(\ul{i})\,(\ul{j})}
\end{align}
for $\mu=0,\dots,N-1$, $\nu=0,\dots,N-2$.
Note that 
$J_{\mu\nu}^{(p|\,\ul{i}\,|\ul{j})} = 0
\Rightarrow
J_{\mu(N-1)}^{p\,\ul{i}} = 0$.
Since $H_N$ is an FT$N$S
candidate symmetrizer and due to the fact that certain
symmetries hold for $H_N$ and $\mathcal A_N$
the $T_{\mu\nu}$ satisfy
\begin{align}
\label{eq:FTNS:sym_T}
T_{\mu\nu}^{(p\,\ul{i}\,\ul{j})} &= T_{\nu\mu}^{\dag (p\,\ul{j}\,\ul{i})},&
T_{\mu\nu}^{p\,\ul{i}\,\ul{j}} &= T_{\mu\nu}^{p\,(\ul{i})\,(\ul{j})}
\end{align}
for $\mu,\nu=0,\dots,N-1$.

Now, assuming a given $J_{\mu\nu}$ which satisfies~\eqref{eq:FTNS:Jmunu_symmetry}
we can easily calculate the reduction variables ${\bar D}_N$ by multiplication
of $J_{N}$ from the left with $H_N^{-1}$ (which exists because $H_N$ is
positive definite by assumption).

Hence, the existence of a first order reduction with candidate symmetrizer
$H_1$ is shown if we prove that there exist~$J_{\mu\nu}$ which 
satisfy~\eqref{eq:FTNS:conserv_TJ_comp} and~\eqref{eq:FTNS:Jmunu_symmetry}
given~\eqref{eq:FTNS:sym_T} holds.

One approach for the proof of this statement is the following.
One defines
$
V_{\mu\nu}^{p\,\ul{i}\,\ul{j}} :=
T_{\mu\nu}^{p\,\ul{i}\,\ul{j}}
-T_{\nu\mu}^{\dag\,p\,\ul{j}\,\ul{i}},
$
which satisfies
$
V_{\mu\nu}^{\dag\, p\,\ul{i}\,\ul{j}} = -V_{\nu\mu}^{p\,\ul{j}\,\ul{i}},
V_{\nu\mu}^{(p\,\ul{j}\,\ul{i})} = 0,
$
and~
$
V_{\nu\mu}^{p\,\ul{j}\,\ul{i}} = V_{\nu\mu}^{p\,(\ul{j})\,(\ul{i})}.
$
Then one uses the ansatz 
\begin{align}
J_{\mu\nu}^{p\,\ul{i}\,\ul{j}} =
\sum_{\pi\in S_{(2N-\mu-\nu-1)}} x_\pi V_{\mu\nu}^{\pi(p)\,\pi(\ul{i})\,\pi(\ul{j})}\nonumber
\end{align}
in equations~\eqref{eq:FTNS:conserv_TJ_comp} and~\eqref{eq:FTNS:Jmunu_symmetry}
to get a linear system for the coefficients~$x_\pi$. If one can show that 
this linear system has a solution then the existence of a first order 
reduction with candidate symmetrizer~$H_1$ follows with the arguments 
given above. This procedure is shown for~$N=3$ 
in {\tt automatic\_construction\_of\_J.nb}
\footnote{{\tt http://www.tpi.uni-jena.de/\~{}hild/FTNS.tgz}}, and we performed the 
same calculations for~$N\leq \maxN$ using computer algebra. 

For arbitrary~$N$ 
the number of coefficients increases like~$N!$. Although many of them can be 
considered redundant because of the symmetries of~$V_{\mu\nu}$ and~$J_{\mu\nu}$, 
the construction of the linear system for the~$x_\pi$ is difficult for 
arbitrary~$N$. Therefore we leave this question open.

\paragraph*{Connection to energy conservation. } Given an FT$N$S 
symmetric hyperbolic system it is straightforward to show that 
the quantity
$E := \int d^3x\,u_{\ul{i}}^\dag H_N^{\ul{i}\ul{j}}u_{\ul{j}}$
is a conserved energy in the principal part, i.e. $E>0$ and
$\p_t E \simeq 0$. To show this one uses the positivity of
$H_N^{\ul{i}\ul{j}}$ and the equations of motion
\eqref{eq:system_FTNS_app1} together with integration by parts:
\begin{align}
\p_t &E \simeq
\frac12\sum_{\mu,\nu=0}^{N-1} (-1)^{N-\mu-1} \int d^3x\,v^{\dag\,\mu}V_{\mu\nu}^{p\,\ul{i}\,\ul{j}}
\p^{(2N-\mu-\nu-1)}_{p\ul{i}\ul{j}}v^\nu= 0,\nonumber
\end{align}
where
\begin{align}
\dot\epsilon_{\mu\nu} &:=
\left(\p^{(N-\mu-1)}_{\ul{i}}v^{\dag\,\mu}\right)
T_{\mu\nu}^{p\,\ul{i}\,\ul{j}}
\;\p^{(N-\nu)}_{p\ul{j}}v^\nu
+
\left(\p^{(N-\mu)}_{p\ul{i}}v^{\dag\,\mu}\right)
T_{\nu\mu}^{\dag\,p\,\ul{j}\,\ul{i}}
\;\p^{(N-\nu-1)}_{\ul{j}}v^\nu\nonumber
\end{align}
and
$
\p^{(2N-\mu-\nu-1)}_{p\ul{i}\ul{j}}v^\nu
= \p_{p}\p_{i_1}\dots\p_{i_{N-\mu-1}}\p_{j_1}\dots\p_{j_{N-\nu-1}}v^\nu.
$
Since $v^\mu$ is an arbitrary solution of the equations of motion
this implies that there exist fluxes $\phi^p_{\mu\nu}$ such that
$\dot\epsilon_{\mu\nu}=\p_p \phi^p_{\mu\nu}\quad\forall\mu,\nu=0,\dots,N-1$.
The existence of a symmetric hyperbolic first order reduction
with the symmetrizer~$H_1^{\ul{i}\ul{j}}$ means that there
exist fluxes~$\phi_{\mu\nu}^p$ of the form
\begin{align}
\phi^p_{\mu\nu} &= \left(\p^{(N-\mu-1)}_{\ul{i}}v^{\dag\,\mu}\right)
F_{\mu\nu}^{p\,\ul{i}\,\ul{j}}
\;\p^{(N-\nu-1)}_{\ul{j}}v^\nu,\nonumber
\end{align}
with
$
F_{\mu\nu}^{p\,\ul{i}\,\ul{j}}
=
J_{\mu\nu}^{p\,\ul{i}\,\ul{j}}
+T_{\mu\nu}^{p\,\ul{i}\,\ul{j}},
$
i.e. that the $v^\mu$ appear in the fluxes only with $(N-\mu-1)$
derivatives.

\section{Conclusion}
\label{section:Conclusion}

We described how the existing notion of strong hyperbolicity
for first and second order in space evolution
equations~\cite{GusKreOli95,GunGar05} can be extended to
FT$N$S systems, i.e. evolution equations of arbitrary
spatial order. The definitions of FT$N$S strong and symmetric 
hyperbolicity allow for the direct construction of well-posed 
initial (boundary) value problems for systems of higher order.

This extension is achieved by proposing a reasonable definition
of strong hyperbolicity for FT$N$S systems and showing that
this new definition can be reduced to the lower order equivalent.
The proof is performed with the help of an iterative differential
reduction of the FT$N$S system from arbitrary to first order.
One finds that an evolution system is FT$N$S strongly hyperbolic
if and only if there exists a first order reduction which is
strongly hyperbolic in the standard first order
sense.

We also considered symmetric hyperbolicity of FT$N$S systems.
In this case one finds that it is better to introduce a direct
reduction to first order instead of using the iterative method
applied to prove statements about strong hyperbolicity. We 
proposed a definition of FT$N$S symmetric hyperbolicity and were 
able to show for $N\leq \maxN$ that it is equivalent to the 
existence of a direct first order reduction which is symmetric 
hyperbolic in the standard first order sense. For higher orders 
we were not successful in showing equivalence, but only one 
direction, that the existence of a symmetric hyperbolic first 
order reduction implies FT$N$S symmetric hyperbolicity.

There are various questions which can be addressed in further
analysis. One is that the proofs about strong hyperbolicity
rely strongly on three spatial dimensions, because the
Levi-Civita symbol $\varepsilon_{ijk}$ is used. Whether
a similar construction is possible for other spatial
dimensionality is not known. For symmetric hyperbolicity the
spatial dimensionality is not used in the calculations, i.e.
the results apply to any dimension. However, as mentioned above,
equivalence for~$N>\maxN$ is not yet shown.

Finally, it is essential for the construction of approximate
solutions to identify good numerical methods. Therefore
it is also of interest to analyze the connection between high
order hyperbolicity and e.g. stability of finite difference
methods.

\section*{Acknowledgments}

The authors wish to thank Carsten Gundlach and Milton Ruiz for 
stimulating discussions. This work was supported in part by 
DFG grant SFB/Transregio~7 ``Gravitational Wave Astronomy''. 

\bibliography{refs}{}

\begin{thebibliography}{10}

\bibitem{Agr72}
M.~S. Agranovich.
\newblock Theorem on matrices depending on parameters and its applications to
  hyperbolic systems.
\newblock {\em Functional Analysis and Its Applications}, 6:85--93, 1972.

\bibitem{Bei04}
R. Beig.
\newblock {Concepts of hyperbolicity and relativistic continuum mechanics}.
\newblock {\em Lect.Notes Phys.}, 692:101--116, 2006.

\bibitem{BeySar04}
H. Beyer and O. Sarbach.
\newblock {O}n the well posedness of the {B}aumgarte-{S}hapiro-
  {S}hibata-{N}akamura formulation of {E}instein's field equations.
\newblock {\em Phys. Rev. D}, 70:104004, 2004.

\bibitem{Bey05}
H.~R. Beyer.
\newblock {Beyond partial differential equations}.
\newblock LNM 1898, 2005.
\newblock  Springer, Berlin, 2007

\bibitem{CalHinHus05}
G. Calabrese, I. Hinder, and S. Husa.
\newblock Numerical stability for finite difference approximations of
  {E}instein's equations.
\newblock {\em J. Comp. Phys.}, 218:607--634, 2005.

\bibitem{Chr00}
D. Christodolou.
\newblock {\em The Action Principle and Partial Differential Equations.}
\newblock Annals of Mathematics Studies, 146. Princeton University Press, 2000.

\bibitem{Ger96}
R.~P. Geroch.
\newblock Partial differential equations of physics.
\newblock In {\em General Relativity}. G.~Hall, editor. 1996.

\bibitem{GunGar04a}
C.~Gundlach and J.~M. Mart\'in-Garc\'ia.
\newblock Symmetric hyperbolicity and consistent boundary conditions for
  second-order {E}instein equations.
\newblock {\em Phys. Rev. D}, 70:044032, 2004.

\bibitem{GunGar04}
C.~Gundlach and J.M. Martin-Garcia.
\newblock Symmetric hyperbolic form of systems of second-order evolution
  equations subject to constraints.
\newblock {\em Phys. Rev. D}, 70:044031, 2004.

\bibitem{GunGar05}
C. Gundlach and J.M. Mart{\'\i}n-Garc{\'\i}a.
\newblock Hyperbolicity of second-order in space systems of evolution
  equations.
\newblock {\em Class. Quantum Grav.}, 23:S387--S404, 2006.

\bibitem{GunGar06}
C. Gundlach and J.M. Martin-Garcia.
\newblock Well-posedness of formulations of the {E}instein equations with
  dynamical lapse and shift conditions.
\newblock {\em Phys. Rev. D}, 74:024016, 2006.

\bibitem{Gus08}
B. Gustafsson.
\newblock {\em High Order Difference Methods for Time Dependent PDE}.
\newblock Springer-Verlag, Berlin, Heidelberg, 2008.

\bibitem{GusKreOli95}
B. Gustafsson, H.-O. Kreiss, and J. Oliger.
\newblock {\em Time dependent problems and difference methods}.
\newblock Wiley, New York, 1995.

\bibitem{KidSchTeu01}
L.~E. Kidder, M.~A. Scheel, and S.~A. Teukolsky.
\newblock Extending the lifetime of 3{D} black hole computations with a new
  hyperbolic system of evolution equations.
\newblock {\em Phys. Rev. D}, 64:064017, 2001.

\bibitem{KreOrtPet10}
H.-O. {Kreiss}, O.~E. {Ortiz}, and N.~A. {Petersson}.
\newblock {Initial-boundary value problems for second order systems of partial
  differential equations}.
\newblock arXiv:1012.1065, December 2010.

\bibitem{Kre70}
H.-O. Kreiss.
\newblock Initial boundary value problems for hyperbolic systems.
\newblock {\em Comm. Pure Appl. Math.}, 23:277--298, 1970.

\bibitem{KreLor89}
H.-O. Kreiss and J.~Lorenz.
\newblock {\em Initial-boundary value problems and the {N}avier-{S}tokes
  equations}.
\newblock Academic Press, New York, 1989.

\bibitem{KrePetYst02}
H.-O. Kreiss, N.~A. Petersson, and Jacob Ystr{\"o}m.
\newblock Difference approximations for the second order wave equation.
\newblock {\em SIAM J. Numer. Anal.}, 40:1940--1967, 2002.

\bibitem{LinSchKid05}
L. Lindblom, M.~A. Scheel, L.~E. Kidder, R. Owen, and O. Rinne.
\newblock A new generalized harmonic evolution system.
\newblock {\em Class. Quantum Grav.}, 23:S447--S462, 2006.

\bibitem{Met00} 
G. M\'etivier.
\newblock The block structure condition for symmetric hyperbolic systems.
\newblock {\em Bulletin of the London Mathematical Society}, 32:689--702, 2000.

\bibitem{NagOrtReu04}
G.~Nagy, O.~E. Ortiz, and O.~A. Reula.
\newblock Strongly hyperbolic second order {E}instein's evolution equations.
\newblock {\em Phys. Rev. D}, 70:044012, 2004.

\bibitem{SarCalPul02}
O.~Sarbach, G.~Calabrese, J.~Pullin, and M.~Tiglio.
\newblock Hyperbolicity of the {BSSN} system of {E}instein evolution equations.
\newblock {\em Phys. Rev. D}, 66:064002, 2002.

\bibitem{SarTig12}
O. Sarbach and M. Tiglio.
\newblock Continuum and discrete initial-boundary value problems and einstein's
  field equations.
\newblock {\em Living Reviews in Relativity}, 15(9), 2012.

\bibitem{Tay81}
M.~E. Taylor.
\newblock {\em Pseudodifferential operators / Michael E. Taylor}.
\newblock Princeton University Press, Princeton, N.J., 1981.

\bibitem{Wal84a}
R.~M. Wald.
\newblock {\em General Relativity}.
\newblock University of Chicago Press, Chicago, 1984.

\end{thebibliography}
\bibliographystyle{plain}


\end{document}